\documentclass[12pt,reqno]{amsart}
\usepackage{amssymb}
\usepackage{graphicx}
\usepackage{amsmath}
\usepackage{a4wide}
\usepackage{version}
\usepackage{mathrsfs}  
\usepackage{tikz}
\usepackage{mathtools}
\usetikzlibrary{arrows,decorations.markings}
\usepackage{hyperref}

\newtheorem{thm}{Theorem}
\newtheorem{lem}[thm]{Lemma}
\newtheorem{prop}[thm]{Proposition}
\newtheorem{cor}[thm]{Corollary}

\theoremstyle{definition}
\newtheorem{definition}[thm]{Definition}
\newtheorem{example}[thm]{Example}
\newtheorem{rmk}[thm]{Remark}

\renewcommand{\k}{\mathbf{k}}
\renewcommand{\t}{\mathbf{t}}

\newcommand{\K}{\mathbb{K}}
\renewcommand{\d}{\partial}

\renewcommand{\r}{\mathfrak{r}} 
\renewcommand{\H}{\mathfrak{H}}

\title{Fukaya categories of plumbings and multiplicative preprojective algebras}
\author{Tolga Etg\"u}  
\address{Department of Mathematics, Ko\c{c} University, Sariyer, Istanbul 34450 TURKEY}
\email{tetgu@ku.edu.tr}
\author{Yank\i\ Lekili}
\address{King's College London, Strand, London, UK}
\email{yanki.lekili@kcl.ac.uk}
\date{} 

\begin{document}

\begin{abstract} Given an arbitrary graph $\Gamma$ and non-negative integers $g_v$ for each vertex
    $v$ of $\Gamma$, let $X_\Gamma$ be the Weinstein $4$-manifold
    obtained by plumbing copies of $T^*\Sigma_v$ according to this graph, where $\Sigma_v$ is a
    surface of genus $g_v$. We compute the wrapped Fukaya
    category of $X_\Gamma$ (with bulk parameters) using Legendrian surgery extending our previous
    work \cite{EL} where it was assumed that $g_v=0$ for all $v$ and $\Gamma$ was a tree. The resulting
    algebra is recognized as the (derived) multiplicative preprojective algebra (and its higher
    genus version) defined by
    Crawley-Boevey and Shaw \cite{C-BS}. Along the way, we find a smaller model for the internal DG-algebra of Ekholm-Ng \cite{EkNg} associated to $1$-handles in the Legendrian
    surgery presentation of Weinstein $4$-manifolds which might be of independent interest.     
\end{abstract}

\maketitle

\section{Introduction}

In this paper, we continue our study from \cite{EL} of computing explicit models for the wrapped
Fukaya category $\mathcal{W}(X_\Gamma)$ of the 4-dimensional open symplectic manifold
$X_\Gamma$, obtained by plumbing cotangent bundles of surfaces according to a graph $\Gamma$.
Previously in \cite{EL}, we restricted to case of plumbings of cotangent bundles of 2-spheres and $X_\Gamma$ was exhibited as a Legendrian surgery along a link $L_\Gamma
\subset (S^3, \xi_{std})$ and the Chekanov-Eliashberg DG-algebra
$\mathscr{B}_\Gamma$ of $L_\Gamma$ was computed explicitly for any tree $\Gamma$. 
It is known that $\mathscr{B}_\Gamma$ is
$A_\infty$-quasi-isomorphic to the endomorphism algebra of the full subcategory of $\mathcal{W}(X_\Gamma)$ whose objects are the cocores of the 2-handles in the surgery presentation \cite{BEE} (cf. \cite[Thm. 2]{EkLe}).  Furthermore, a combination of results of \cite{BEE} and \cite{Ab} implies that these cocores split-generate $\mathcal{W}(X_\Gamma)$.
A proof of the stronger statement of generation (rather than split-generation) is recently given by Chantraine, Dimitroglou-Rizell, Ghiggini, Golovko \cite{CDGG}. Thus, we take $\mathscr{B}_\Gamma$ as our primary object of interest.

We improve on the results of \cite{EL} in two aspects. Firstly, no restriction is imposed on $\Gamma$ which was assumed to be a tree in
\cite{EL} and only plumbings of $T^*S^2$'s were considered. When $\Gamma$ is not a tree or cotangent bundles of positive genus surfaces are used, Legendrian surgery description of $X_\Gamma$ necessarily includes
$1$-handles, and the computation of $\mathscr{B}_\Gamma$ is more involved. More
precisely, for any connected graph $\Gamma$ with $s$ vertices and $s+k-1$ edges and $g= \sum_v g_v$, we shall construct a Legendrian
$L_\Gamma$ in $\#^{k+2g}(S^1 \times S^2, \xi_{std})$ with $s$ components such that
$X_\Gamma$ is obtained by Legendrian surgery on $L_\Gamma$.  A roadmap for an explicit computation of $\mathscr{B}_\Gamma$
is provided by the work of Ekholm-Ng \cite{EkNg} which involves a countably generated internal DG-algebra $\mathscr{I}$ associated to each 1-handle
in the surgery presentation. We 
find a simplification by proving that a finitely generated DG-subalgebra of
$\mathscr{I}$ is quasi-isomorphic to $\mathscr{I}$. 

Secondly, when $\Gamma$ is one of the $A_n$- and $D_n$-type trees (and conjecturally we also include
$E_6,E_7,E_8$-types), it was shown in \cite{EL} that $\mathscr{B}_\Gamma$ is quasi-isomorphic to
Ginzburg's DG-algebra $\mathscr{G}_\Gamma$, away from a small set of exceptional characteristics for the ground
field $\mathbb{K}$.  On the other hand, if $\Gamma$ is a non-Dynkin tree, the explicit computation in \cite{EL} shows that
$\mathscr{B}_\Gamma$ is a deformation
of $\mathscr{G}_\Gamma$. In this case, it is also known that
$\mathscr{G}_\Gamma$ is
quasi-isomorphic to $H^0(\mathscr{G}_\Gamma)= \Pi_\Gamma$, the (additive)
preprojective algebra attached to $\Gamma$ (in contrast, $\mathscr{G}_\Gamma$ has non-trivial cohomology in every non-positive degree,
when $\Gamma$ is a Dynkin tree). It was argued that the associated formal deformation is trivial over a
field of characteristic 0 since $\Pi_\Gamma$ has no formal deformations as a Calabi-Yau algebra, but the precise determination of quasi-isomorphism type of
$\mathscr{B}_\Gamma$ was not given.  

In this paper, we first restrict to the case $g_v=0$ and show that for any finite graph $\Gamma$,  $\mathscr{B}_\Gamma$ is quasi-isomorphic to the derived multiplicative preprojective algebra
$\mathscr{L}_\Gamma$ whose zeroth cohomology is isomorphic to the multiplicative preprojective
algebra $\Lambda_\Gamma$ defined in \cite{C-BS} by Crawley-Boevey and Shaw with motivation coming
from the Deligne-Simpson problem. Next, we extend this to the case where $g_v$ is
arbitrary. The resulting DG-algebra is quasi-isomorphic to a positive genus version of the derived multiplicative
algebra that appears in \cite{BK}. We shall denote this by $\mathscr{L}_{\Gamma, {\mathbf g}}$
and its zeroth cohomology $H^0(\mathscr{L}_{\Gamma, \mathbf{g}})$ by
$\Lambda_{\Gamma, \mathbf{g}}$, where $\mathbf{g}= (g_1,g_2,\ldots,g_s)$ is an $s$-tuple of non-negative integers recording the genus $g_v$ associated to
vertices of $\Gamma$. When $\mathbf{g}$ is omitted, it is understood that $g_v=0$ for all $v$. 

Before giving the definition of $\Lambda_{\Gamma, \mathbf{g}}$ and
$\mathscr{L}_{\Gamma, \mathbf{g}}$, let us first introduce some notation. 
As in \cite{EL}, we fix a field $\mathbb{K}$ and the semisimple ring 
\[ \k \coloneqq \bigoplus_{v=1}^{s} \mathbb{K} e_v \]
where $e_v^2 = e_v$ and $e_v e_w =0$ for $v\neq w$, and $\{ 1, \ldots , s\}$ is the set of vertices of $\Gamma$. 
Here, we shall work over a bigger ring:
\[ \k\langle \t^\pm \rangle = \k \langle t_1, t_1^{-1}, t_2, t_2^{-1},\ldots, t_s, t_s^{-1}
\rangle / (t_v t_v^{-1} = t_v^{-1} t_v = e_v , t^{\pm}_v e_w = e_w t_v^{\pm} = 0 \text{\ \ for\ \ } v \neq w
)\] 
We also fix a spanning tree $T \subset \Gamma$ and orient the edges of $\Gamma$. The multiplicative preprojective algebra $\Lambda_\Gamma$ is the associative algebra obtained from $\k \langle \t^\pm \rangle$ by adjoining free (non-commuting) generators
\begin{align*} 
    c_a, c_a^* & \ \ \text{for any arrow\ \ } a \in \Gamma , \text{ \ and \ }\\
    z_a, z_a^{-1} & \  \text{ for any arrow \  } a \in \Gamma \backslash T \\
    \alpha_{v,i}, \alpha_{v,i}^{-1}, \beta_{v,i}, \beta_{v,i}^{-1} & \ \text{ for any vertex \ } v \in \{1, \dots , s\} \text{ and } i=1,2,\ldots, g_v. 
\end{align*}
(Note that there are indeed 4 generators associated to $a \in \Gamma \backslash T$: $c_a,c_a^*, z_a,z_a^{-1}$.)

The $\k \langle \t^{\pm} \rangle$-bimodule structure on $\Lambda_\Gamma$ is obtained by extending the scalars from the
$\k$-bimodule structure defined by 
\begin{align*}
e_w c_a e_v &= c_a,\ \ e_v c_a^* e_w =c_a^* , \\
e_v & z_a^{\pm} e_v = z_a^{\pm} 
\end{align*}
if and only if $v=s(a)$ and $w=t(a)$ are the source and target vertices of the arrow $a$,
respectively, and $\k$-bimodule structure on the remaining generators is the obvious one:
\begin{align*}
    e_v & \alpha_{v,i}^{\pm} e_v = \alpha_{v,i}^{\pm} \\
    e_v & \beta_{v,i}^{\pm} e_v = \beta_{v,i}^{\pm}
\end{align*}
In all other cases, the product with an idempotent gives zero. With this understood, the underlying algebra structure of $\Lambda_\Gamma$ is the tensor algebra of the $\k$-bimodule generated by $c_a,c_a^*, z_a^{\pm}, \alpha_{v,i}^\pm, \beta_{v,i}^\pm, t_v^{\pm}$. 

Note that the variables $t_v$ are taken to be non-central (unless we specialise). 
The generators of $\Lambda_\Gamma$ are subject to the following relations (all multiplications are read from right-to-left): 
\begin{align*}  
    z_a z_a^{-1} = z_a^{-1} z_a = e_v , \ \ &\text{\ \ for  } a \in \Gamma \backslash T \text{\ and\ }  v=s(a) \\
    e_v + c_a^*c_a = z_a , \ \ &\text{\ \ for } a \in \Gamma \backslash T \text{\ and\ }  v=s(a)  \\
    \prod_{t(a)=v}(e_v + c_ac_a^*)=  t_v \prod_{i=1}^{g_v} [\alpha_{v,i}, \beta_{v,i}] \prod_{s(a)=v} (e_v + c_a^*c_a),  \ \ &\text{\ \ for } v \in \{1,\ldots,s\} 
\end{align*} 
where the terms in the products are ordered according to a fixed order of all the arrows in $\Gamma$,
and the empty product is understood to be the relevant idempotent. The notation $[x,y]$
stands for the commutator $xyx^{-1}y^{-1}$. 

It is known that $\Lambda_{\Gamma,\mathbf{g}}$ is independent of the orientation and order of the edges of
$\Gamma$, and the choice of the spanning tree $T$ (cf. \cite[Thm. 1.4]{C-BS}). This also
follows from our main result by the invariance properties of Chekanov-Eliashberg
DG-algebra $\mathscr{B}_{\Gamma, \mathbf{g}}$. 

A DG-algebra $\mathscr{L}_{\Gamma,\mathbf{g}}$ such that $H^0(\mathscr{L}_{\Gamma,\mathbf{g}}) \simeq
\Lambda_{\Gamma,\mathbf{g}}$ is
discussed in \cite[Rem.5C \& Ex.5.6]{BK}  (where $t_\nu$ are called $q_\nu$ which are taken as parameters in $\mathbb{C}^\times$, see also \cite{yeung}) under the name (universal) derived
multiplicative preprojective algebra. In the current set-up, $\mathscr{L}_{\Gamma,\mathbf{g}}$ can
be defined by adding generators $\zeta_a$ for $a \in \Gamma \setminus T$ and $\tau_v$ for $v\in \{1 , \dots , s\}$ with grading $-1$, in lieu of the last two relations above, to the
set of generators of $\Lambda_{\Gamma,\mathbf{g}}$ (which are assigned grading $0$) so that
$H^0(\mathscr{L}_{\Gamma,\mathbf{g}}) = \Lambda_{\Gamma,\mathbf{g}}$. 
Namely, the differentials of these additional generators are
    \begin{align*}\d \zeta_a &= e_v+c_a^*c_a - z_a, \ \ \text{\ \ for } a \in \Gamma \backslash T \text{\ and\ }
    v=s(a) \\
        \d \tau_v &= \prod_{t(a)=v}(e_v + c_ac_a^*) - t_v \prod_{i=1}^{g_v} [\alpha_{v,i},\beta_{v,i}] \prod_{s(a)=v} (e_v + c_a^*c_a), \ \  \text{\ \ for } v
    \in \{1,\ldots,s\} \end{align*}
 
We can now state our main theorem succinctly as follows:

\begin{thm}\label{mainthm} For any graph $\Gamma$ and $\mathbf{g}$, there is a quasi-isomorphism
    between the derived multiplicative preprojective algebra $\mathscr{L}_{\Gamma,\mathbf{g}}$ and the
    Chekanov-Eliashberg DG-algebra $\mathscr{B}_{\Gamma, \mathbf{g}}$.
  \end{thm}

\begin{rmk} The main theorem is proved by identifying $\mathscr{B}_{\Gamma,\mathbf{g}}$ and
    $\mathscr{L}_{\Gamma,\mathbf{g}}$ as DG-algebras
    by direct computation of the Chekanov-Eliashberg DG-algebra of the Legendrian
    $L_{\Gamma,\mathbf{g}}$. In
    view of the surgery results from \cite{BEE}, \cite{EkLe}, our result provides
    computations of various wrapped Fukaya categories via an $A_\infty$ quasi-isomorphism
    \[ \mathscr{B}_{\Gamma,\mathbf{g}}  \cong \bigoplus_{v,w} CW^*(C_v,C_w) \]  
where $CW^*(C_v,C_w)$ stands for the morphisms between the cocores
    $C_v,C_w$ in the partially wrapped Fukaya of $(\#^{k+2g} (S^1 \times \mathbb{D}^3), \omega_{std})$
    where all Legendrians $L_v$ are used as stops (see \cite[Conj. 3]{EkLe}). In this setting, it
    is important that the indeterminates $t_v$'s do not commute with the other generators. On the other
    hand, if one specialises
    to $t_v = e_v$ for $v=1,\ldots,s$, one recovers the fully wrapped Fukaya category of $X_\Gamma$ where $X_\Gamma$ is equipped with
    the exact symplectic form associated with its Weinstein structure. Setting the value of
    $\mathrm{Log}(t_v) = \omega([S_v])$
    to other constants, has the effect of turning on non-exact (bulk) deformations of the symplectic form,
    by giving area to the zero-sections $S_v$ of the cotangent bundles involved in the plumbing $X_\Gamma$. One can interpret this as a computation
    of bulk deformed wrapped Fukaya categories. 
\end{rmk}

\begin{rmk} As in \cite{EkLe}, let 
    \[ \mathscr{A}_{\Gamma,\mathbf{g}} = \bigoplus_{v,w} CF^*(S_v,S_w) \] be the endomorphism algebra of the
    compact exact Lagrangians in $X_\Gamma$ which form the Lagrangian
    skeleton of the plumbing $X_\Gamma$. After letting $t_v=e_v$, for $v=1,\ldots,s$, one has a
    canonical augmentation 
    \[ \epsilon : \mathscr{B}_{\Gamma,\mathbf{g}} \to \k \]
corresponding to the module $\bigoplus CF^*(C_v,S_v)$ in the (fully) wrapped Fukaya category of
    $\mathcal{W}(X_\Gamma)$, given explicitly by
    \[ \epsilon(c_a) = \epsilon(c_a^*) = 0\  ,\  \epsilon(\alpha^{\pm}_{v,i}) =
    \epsilon(\beta^{\pm}_{v,i})=  e_v = \epsilon(z^{\pm}_a),\  v =s(a) \]
    By \cite[Thm. 4]{EkLe}, one can compute 
    \[ \mathscr{A}_{\Gamma, \mathbf{g}} \simeq
    \mathrm{Rhom}_{\mathscr{B}_{\Gamma,\mathbf{g}}}(\k,\k) \]
This could be useful in studying Fukaya categories of symplectic 4-manifolds.
    Namely, if one has any configuration of transversely intersecting compact Lagrangians in a symplectic 4-manifold whose
    intersection pattern is encoded by the decorated graph ($\Gamma$,$\mathbf{g}$), the corresponding $A_\infty$-algebra is a
    deformation of $\mathscr{A}_{\Gamma,\mathbf{g}}$.

    We also note that the grading in $\mathscr{B}_{\Gamma,\mathbf{g}}$ is supported in non-positive degrees. In
    higher dimensional plumbings, Abouzaid and Smith used this property and an additional requirement
    on the zeroth cohomology of cocores \cite[Prop. 4.2]{AS} that guarantees that the only
    modules supported in a single degree are given by the compact cores with local systems, to deduce
    that the compact cores of the plumbing generate the compact Fukaya category. In our case, this
    additional requirement does not hold. Nonetheless, the core Lagrangian surfaces $\{ S_v \}$ are known to generate the compact Fukaya category
    $\mathcal{F}(X_\Gamma)$ in several situations, including when $\Gamma$ is of Dynkin type with
    $\mathbf{g}=0$ (\cite[Lem. 4.15]{Seidelgraded} and
    \cite[Cor. 5.8]{Se}). We also remark that Keating \cite{keating} gave examples of Weinstein
    4-manifolds which are \emph{partial compactifications} of plumbings with $\mathbf{g}=0$, where an exact Lagrangian
    torus (equipped with certain local systems) is not generated by the core spheres. 
\end{rmk}

\begin{rmk} After we wrote our first paper \cite{EL} on computing wrapped Fukaya categories of
    $X_\Gamma$, which studied the case where $\Gamma$ is a tree and $\mathbf{g}=0$, an interesting paper by
    Bezrukavnikov and Kapranov \cite{BK} appeared, where finite dimensional
    $\Lambda_{\Gamma,\mathbf{g}}$-modules are related to microlocal sheaves on the Lagrangian skeleton of $X_\Gamma$ given by the
    union of transversely intersecting surfaces. In view of the  correspondence
    between Fukaya
    categories of cotangent bundles and microlocal sheaves on the zero-section \cite{NZ}, Thm.~(\ref{mainthm}) provides
    a confirmation of the more general expectation of an equivalence between Fukaya categories of Weinstein
    manifolds and microlocal sheaves on (singular) Lagrangian skeleta for the manifolds $X_\Gamma$. This more general equivalence is
    the subject of ongoing work by Ganatra, Pardon and Shende. 
    \end{rmk}

Let us consider the case $\Gamma$ is a tree and $\mathbf{g}=0$, and specialise to $t_v=e_v$. 
By definition of $\mathscr{B}_\Gamma$, the differential $\partial \tau_v$ consists of words of (even) length
    $\geq 2$. If we truncate it so that only words of length 2 appears, the resulting
    DG-algebra is known as the Ginzburg DG-algebra $\mathscr{G}_\Gamma$ which was studied in
    \cite{EL}. It is known that if $\Gamma$ is not of Dynkin type, $\mathscr{G}_\Gamma$
    has cohomology only in degree 0 (cf. Thm 8, Cor. 27 \cite{EL}) and $H^0
    (\mathscr{G}_\Gamma)= \Pi_\Gamma$ is the additive preprojective algebra. Now, by considering the
    word-length filtration, we get a spectral sequence from $\mathscr{G}_\Gamma$ to
    $\mathscr{B}_\Gamma$. At least when $\mathrm{char} \mathbb{K}=0$, one can show that the length spectral
sequence used in the above proof is degenerate. The reason for this was given in \cite[pg.
7]{EL}. Namely, it is known that $\Pi_\Gamma$ has no formal deformations as a Calabi-Yau algebra.
However, the length spectral sequence does not converge strongly if $\Gamma$ is non-Dynkin, hence one cannot conclude that
$\mathscr{B}_\Gamma$ and $\mathscr{G}_\Gamma$ are quasi-isomorphic without having to complete both
algebras. Indeed, for non-Dynkin $\Gamma$, the additive preprojective algebra
$\Pi_\Gamma = H^0(\mathscr{G}_\Gamma)$ and the multiplicative preprojective algebra $\Lambda_\Gamma =
H^0(\mathscr{B}_\Gamma)$ are not isomorphic. This can be seen by considering the representation
varieties of $\Pi_\Gamma$ and $\Lambda_\Gamma$ - additive and multiplicative Nakajima quiver
varieties. Thus, it is the (derived) multiplicative preprojective algebra
$\mathscr{B}_\Gamma$
that is central to the study of wrapped Fukaya categories of plumbings, not the additive preprojective algebra (or its derived version).

\begin{example} When $\Gamma$ is just a single vertex labelled with $g$, then $X_\Gamma$ is
    just the cotangent bundle of a genus $g$ surface $\Sigma_g$. In this case, we see that the
    multiplicative preprojective algebra is simply the group-ring
    $\mathbb{K}[\pi_1(\Sigma_g)]$, which agrees with the based loop space homology of the cotangent
    fibre. This agrees with the computation given in \cite{EkNg}.  
\end{example} 

\begin{example} When $\Gamma$ has a single vertex (labelled with $0$) and a single loop, the symplectic manifold
    $X_\Gamma$ obtained by self-plumbing of $T^*S^2$ was studied in detail in
    \cite{pascaleff}. In this case, it is easy to see that (cf. \cite[Ex. 1.3]{C-BS}) the
    multiplicative preprojective algebra (with $t=1$) $\Lambda_\Gamma$ is quasi-isomorphic to the
    commutative algebra $\mathbb{K}[x,y,(1+xy)^{-1}]$, which agrees with the homological
    mirror symmetry statement proven in \cite{pascaleff} (cf. \cite{CM}). This variety has a holomorphic symplectic form given by $\Omega=dx \wedge dy/(1+xy)$ which agrees with the standard holomorphic symplectic form on the two copies of $(\mathbb{C}^\times)^2$ given by the co-ordinates $(x,1+xy)$ (for $x\neq 0$) and $(y,1+xy)$ (for $y \neq 0$). One can consider the quantization of the functions compatible with the Poisson structure induced by $\Omega$. We point out that if $t$ is not specialized to 1 but carried along as a commuting complex parameter corresponding to an area parameter on the immersed sphere, then we get the ``quantized algebra'' where $x$ and $y$ no longer commute but instead satisfy    
    \[ 1+ xy  = t(1+yx) \] 
where $(1+xy)$ is still invertible. Indeed, it is easy to see from this that we get the standard quantum torus algebras on each copy of $(\mathbb{C}^*)^2$. For example, we have
\[ t(1+xy)x = tx(1+yx) = x(1+xy) \]
\end{example}

We end the introduction by pointing out an intriguing connection to (multiplicative) Nakajima quiver
varieties that we plan to return to in the future. Note that such a connection was also observed in
\cite{BK} (see also \cite{ST}) via calculations using microlocal sheaves.

It is well known that the moduli spaces of representations of the
multiplicative preprojective algebra $\Lambda_\Gamma$ has the structure of algebraic symplectic
varieties known as the multiplicative Nakajima quiver varieties (studied by Yamakawa
\cite{yamakawa}, see also Boalch \cite{boalch}). By our
result, any Lagrangian $L \subset X_\Gamma$ gives a representation $\bigoplus_v \mathrm{Hom}(C_v, L)$
of the multiplicative preprojective algebra $\bigoplus_{v,w} \mathrm{Hom}(C_v,C_w)$. Thus, our result can be interpreted as 
the statement that ``moduli space of objects'' in the derived Fukaya category of $X_\Gamma$ can be identified with the corresponding multiplicative Nakajima quiver varieties. These are finite type varieties once one fixes a dimension vector consisting of 
$d_v \in \mathbb{Z}_{\geq 0}$ for each vertex $v$ of $\Gamma$ and a stability condition encoded by
$\theta_v \in \mathbb{Q}$ for each $v$, satisfying $ \sum_v d_v \theta_v= 0$, and the values of $t_v \in \mathbb{C}^\times$. The choice of dimension vectors
corresponds to fixing the rank $d_v = \mathrm{dim}_\mathbb{K} \mathrm{Hom}(L,C_v)$, and the choice of stability
condition should correspond to fixing a choice of Bridgeland stability on the  derived Fukaya category
of $X_\Gamma$.
Finally, note that the symplectic structures on the Nakajima quiver varieties correspond to the fact
that $\mathscr{B}_{\Gamma,\mathbf{g}}$ is 2-Calabi-Yau (cf. \cite{PTVV}). \\

In Section \ref{Ch-El}, after reviewing the construction of Chekanov-Eliashberg algebra associated to a
Legendrian surgery presentation of a Weinstein 4-manifold following Ekholm-Ng \cite{EkNg}, we
provide a simplification that results in a finitely generated DG-algebra even in the presence of
1-handles. In Section \ref{g=0}, we draw a Legendrian surgery presentation of plumbings of $T^*S^2$'s
according to an arbitrary graph and prove our main theorem for $\mathbf{g}=0$. In Section \ref{g>0},
we extend the results of Section \ref{g=0} to arbitrary $\mathbf{g}$. \\

{\bf Acknowledgments:} T.E. is partially supported by the Technological and Research Council of
Turkey through a BIDEB-2219 fellowship. Most of the work was carried out while T.E. was visiting Princeton University.
Y.L. is supported in part by the Royal Society (URF) and
the NSF grant DMS-1509141. We would like to thank Pavel Etingof, Lenhard Ng, and Ivan Smith for helpful correspondence. We also thank referees for their comments. 

\section{Legendrian DG-algebra on a subcritical Weinstein 4-manifold}\label{Ch-El}

Recall that any Weinstein 4-manifold $X$ has a Weinstein handlebody decomposition consisting of 0-,
1- and 2-handles. Assuming $X$ is connected, one can find a handlebody decomposition
with a unique 0-handle. Furthermore, Weinstein $1$-handles can be attached in a unique way, hence
the more interesting part of the handlebody decomposition is the attachment of Weinstein $2$-handles along Legendrian knots in $\#^k (S^1 \times S^2)$ equipped with the
unique Stein-fillable contact structure, where $k$ is the number of $1$-handles. If $k=0$, the
$2$-handles are attached along Legendrian knots in $S^3$ equipped with its standard contact contact
structure. 

Following Gompf \cite{G}, a surgery diagram for $X$ is usually presented as a front diagram in a
\emph{standard form}. In order to give a combinatorial description of the Chekanov-Eliashberg's
Legendrian DG-algebra (\cite{C, E}) associated with a Legendrian link $L \subset \#^k (S^1
\times S^2)$
given in Gompf's standard form, Ekholm and Ng \cite{EkNg} described a procedure called
\emph{resolution} (analogous to the resolution of a Legendrian in $S^3$ in \cite{Ng}). We give an example in Figure \ref{fig0} and refer the readers to the original references \cite{G, EkNg} for precise definitions. 

\begin{figure}[htb!]
\centering
    \begin{tikzpicture}
   \tikzset{->-/.style={decoration={ markings,
                mark=at position #1 with {\arrow{>}}},postaction={decorate}}}

        \begin{scope}[yscale=1.2] 
         \draw[blue,thick] (-1.7,4.4) [in=180,out=0]to (1,4);
        \draw[blue,thick] (-1.7,3.6) [in=180,out=0]to (1,4);
        \draw[blue,thick] (1.7,4.4) [in=0,out=180]to (-1,4);
        \draw[blue,thick] (1.7,3.6) [in=0,out=180]to (-1,4);
        \draw[blue,thick] (-1.7,4.4) to (-2.7,4.4);
         \draw[blue,thick] (-1.7,3.6) to (-2.7,3.6);
          \draw[blue,thick] (1.7,3.6) to (2.7,3.6);
           \draw[blue,thick] (1.7,4.4) to (2.7,4.4);
         
        \draw [thick] (-3,4) circle (0.5); 
 \draw [thick] (3,4) circle (0.5); 
  \draw [thick] (-3,2) circle (0.5); 
 \draw [thick] (3,2) circle (0.5); 

         \draw[blue,thick] (-2.7,2.4) [in=270,out=0]to (0,2);
        \draw[blue,thick] (-2.7,1.6) [in=90,out=0]to (0,2);
        \draw[blue,thick] (2.7,1.6) [in=90,out=180]to (-1.7,2);
        \draw[blue,thick] (2.7,2.4) [in=270,out=180]to (-1.7,2);

           \draw [thick] (-3,0) circle (0.5); 
 \draw [thick] (3,0) circle (0.5); 
    \draw[blue,thick] (-2.7,0.4) [in=260,out=0]to (0.4,0);
        \draw[blue,thick] (-2.7,-0.4) [in=100,out=0]to (0.4,0);

            \draw[blue,thick] (2.7,-0.4) [in=80,out=180]to (-0.3,0);
        \draw[blue,thick] (2.7,0.4) [in=280,out=180]to (-0.3,0);

        \end{scope}

    \end{tikzpicture}
    \caption{A Legendrian knot in Gompf normal form (top), Ekholm-Ng's resolution (middle), A
    simplifying Legendrian isotopy via Reidemeister III (bottom)}
    \label{fig0}
\end{figure}
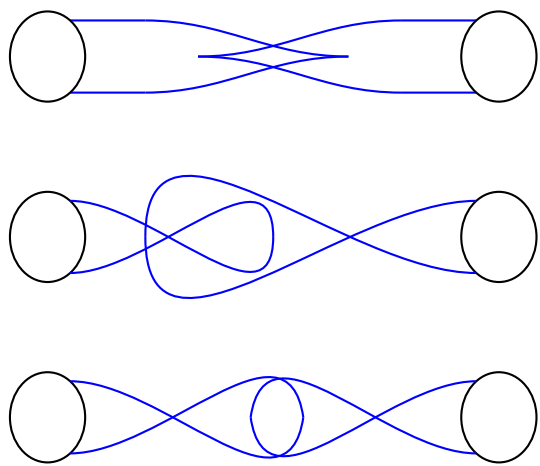

We next recall the description of the Legendrian DG-algebra, here denoted $CE^*(L)$, that
Ekholm and Ng provide coming from a resolution diagram. All our complexes are cohomological,
thus we reverse the gradings from \cite{EkNg}. 

Fix a field $\mathbb{K}$. Let $L= L_1 \cup L_2 \cup \ldots L_s \subset \#^k (S^1 \times S^2)$ be a link given in the resolved
form. 
Let $a_1,\ldots, a_n$ denote the crossings of the resolution diagram. One refers to the set
$\{a_1,\ldots,a_n\}$ as the set of \emph{external generators}. Note each $a_i$ corresponds to
a Reeb chord connecting $L_{v}$ to $L_{w}$ for some $v$ and $w$, where $L_v$ is the label of the
under-strand and $L_w$ is the label of the over-strand at the crossing $a_i$. Thus, we can form the
$\k$-bimodule
\[ \mathcal{E} \coloneqq \bigoplus_{i=1}^n \mathbb{K} a_i\]
where $e_w a_i e_v = a_i$ precisely if $a_i$ corresponds to a Reeb chord from $L_v$ to $L_w$. 

In the interior region of each $1$-handle, we also have generators. These will be referred to as the
\emph{internal generators}. Following Ekholm-Ng, we describe these generators. Let $k$
be the number of 1-handles. For each $1 \leq l \leq k$, let $o_l$ denote the number of strands of
$L$ entering to $l$-th handle. We label the strands on the left of the diagram from top to bottom by
$1, \ldots, o_l$ and on the right of the diagram from bottom to top by $1,\ldots, o_l$ (note the
switch in the ordering). For each $1$-handle, there are infinitely many internal generators. For $1
\leq l \leq k$, these are labelled as follows:
\begin{align*} 
    &c_{ij;l}^0 \text{ for } 1 \leq i <j \leq o_l, \\
    &c_{ij;l}^p \text{ for } p>0 \text{ and } 1 \leq i, j \leq o_l. 
\end{align*}
The generator $c_{ij;l}^p$ for $p\geq 0$ connect the $i^{th}$ strand to $j^{th}$ strand going to
the handle labelled $l$, thus corresponds to a Reeb chord from some $L_v$ to another $L_w$. This defines a $\k$-bimodule structure by defining $e_w c_{ij;l}^p e_v =  c_{ij;l}^p$. 
We write
\[ \mathcal{I} = \mathcal{I}_1 \oplus \mathcal{I}_2 \oplus \ldots \oplus \mathcal{I}_k \]
for the corresponding $\k$-bimodule generated by the internal generators. 

\begin{definition} We let $CE^*(L)$ to be the DG-algebra over $\k$ generated by the external generators
$$a_1,a_2,\ldots, a_n , $$
the internal generators
    \begin{align*}
        &c_{ij;l}^0,\ 1 \leq i < j \leq o_l,\ l=1,\ldots,k\\
        &c_{ij;l}^p,\ p>0,\ 1\leq i, j \leq o_l,\ l=1,\ldots,k, 
    \end{align*} 
and
$$t_1, t_1^{-1}, \ldots, t_s, t_s^{-1} $$
with relations
    \begin{align*} t_v t_v^{-1} &= t_v^{-1} t_v =e_v, \text{ for all } v=1,\ldots,s, \\
        t^\pm_v t^{\pm}_w &=0 , \text{ for } v \neq w
    \end{align*}
and the $\k$-bimodule structure described above.

The differential on the external generators is given by
    \[ \partial(a_i) = \sum_{r \geq 0} \sum_{b_1,\ldots,b_r} \sum_{\Delta \in
    \Delta(a_i;b_1,\ldots,b_r)} \pm \{ t_1^{-n_1(\Delta)}, \ldots ,
    t_{s}^{-n_{s} (\Delta)} , b_r, \ldots, b_1 \} \]
    where $b_q$ are either external generators or of the form $c^{0}_{ij;l}$ and $n_v$ record the
    number of times the boundary of the holomorphic disk $\Delta$ passed through $*_v$, and the
    notation 
    \[ \{ t_1^{n_1(\Delta)}, \ldots ,
    t_{s}^{n_{s} (\Delta)}, b_r, \ldots ,b_1 \} \]
stands for the time-ordered product where reading from right-to-left corresponds to the order of terms appearing in the boundary of $\Delta$ when it is traversed
    counter-clockwise. In particular, note that $t_v$ and $b_q$ maybe interlaced, as we do not
    necessarily insist that these commute. See \cite[Sec. 2.4C]{EkNg} for the determination of
    the sign. \footnote{As in \cite{EL}, \cite{EkLe}, we
    follow the convention that multiplications are read from \emph{right-to-left}. This is not the
    same convention used in \cite{EkNg}. Throughout, our formulae are adjusted accordingly by
    taking the opposite DG-category (see \cite[pg. 8]{EL}). } 
    
        The differential on the internal generators (see Def.~(\ref{defint})) satisfies
    $\d(\mathcal{I}_l) \subset \mathcal{I}_l$ for all $l=1,\ldots, k$ and only depends on the number
    of strands passing through the 1-handle. Thus, the internal DG-algebra corresponding to
    a $1$-handle with $n$ strands is denoted by $\mathscr{I}_n$. 
    
    The differential of a generators of type $t_v$ is trivial.

    It is possible to put a $\mathbb{Z}$-grading on $CE^*$ (by making some extra choices). We refer to \cite[Sec. 2.4B]{EkNg} for a
    detailed discussion of gradings on $CE^*$. Though, we use the negative of the grading given in
    \cite{EkNg} so that the differential increases this grading by 1. Note that the grading
    $|t_v| = 2 r(L_v)$ where $r(L_v)$ is the rotation number.
\end{definition} 

We shall next discuss the internal algebra $\mathscr{I}_n$ of Ekholm-Ng in more detail and show that it can be simplified to only include the free generators
of the form $c_{ij;l}^0$ and $c_{ij;l}^1$. Hence, in particular, it is finitely generated up to quasi-isomorphism.

\subsection{A simplification of Ekholm-Ng's internal DG-algebra}\label{simplify}

In Ekholm-Ng \cite{EkNg}, the internal algebra $\mathscr{I}_n$ was constructed geometrically from the study of
Reeb chords internal to a 1-handle. It has the undesirable feature of being infinitely generated. Here, we provide smaller models for the quasi-isomorphism type of $\mathscr{I}_n$. In particular, we compute the cohomology of
$\mathscr{I}_n$ for all $n$, and show that $\mathscr{I}_n$ is quasi-isomorphic to a finitely generated subalgebra of itself.  

Note that throughout this paper the base ring is taken to be the semisimple ring $\k$ where the number
of idempotents is given by the number of $2$-handles in the surgery diagram corresponding to the
components of the Legendrian link.

In order to achieve full generality, we shall first define $\mathscr{I}_n$ as a DG-algebra over a different ring
\[ \k_n = \bigoplus_{i=1}^n \mathbb{K} e_i \]
As usual, $n$ denotes the number of strands entering the $1$-handle, irrespective of the $2$-handles that these strands belong to.
On the other hand, the internal DG-subalgebra of $CE^*(L)$ for a particular $L$ is obtained from $\mathscr{I}_n$ by a base change
$\k \to \k_n$ given by the configuration of $2$-handles entering in the relevant $1$-handle. More precisely, at each 1-handle, the ring map from $\mathbf{k} \to \mathbf{k}_n$ is given by sending the idempotent $e_v$ associated to a $2$-handle $L_v$ to the sum of idempotents $e_{i_1} + \ldots + e_{i_v}$ in $\mathbf{k}$ where $i_1,\ldots, i_v$ are the strands that pass through the $1$-handle which belong to $2$-handle $L_v$. For example, if a component of $L_v$ enters to the same 1-handle more than once, then a chord that goes between the two strands of $L_v$ can be composed with itself (arbitrarily many times). 

We note that in \cite{EkNg}, $\k$ is always taken to be the ground field $\mathbb{K}$, regardless of the number of components of $L$ entering in the 1-handle. This is the opposite extreme to taking $\k = \k_n$ in that all the chords between any strands are allowed to be composable. However, from the perspective of Legendrian surgery, the correct choice if $L$ is not connected is to work over a ring $\k$ with idempotents in bijection with the connected components of $L$.

We give a formal definition as follows:

\begin{definition} \label{defint} We define $\mathscr{I}_n$ to be a DG-category over $\mathbb{K}$ with $n$-objects : $X_1 , \dots ,
    X_n$, and morphisms $\bigoplus_{i,j} \mathrm{hom} (  X_i , X_j)$ are freely generated (over $\k_n$)
 by
    morphisms in $\mathrm{hom}(X_i, X_j)$ given by:
    \[ c^0_{ij} : 1 \leq i < j \leq n \ \mbox{ and } \  c^p_{ij} : 1 \leq p ,   1\leq i,j \leq n \] 
    The (cohomological) grading is defined by first fixing $n$ integers $(m_1,m_2,\ldots,m_n)$ and then defining:
$$ |c^p_{ij} | = 1-2p+m_j- m_i. $$
    The differential $\d$ is determined by 
    \begin{align*} 
    \partial c^0_{ij} &= \sum_{i<r<j} (-1)^{m_r+m_j} c_{rj}^0 c_{ir}^0, \ \ i < j \\
    \partial c^1_{ij} &= \delta_{ij}+\sum_{r>i} (-1)^{m_r+m_j} c^1_{rj} c^0_{ir} + \sum_{r<j}
    (-1)^{m_r+m_j} c^0_{rj} c^1_{ir}  \\
\partial c^p_{ij} &= \sum_{0 \leq q \leq p}
\sum_r (-1)^{m_r+m_j} c^{p-q}_{rj} c^{q}_{ir} , \ \ p>1 \end{align*}
   ($c^0_{ij}$ is set to be $0$ for $i\geq j$, $\delta_{ij} = e_i = e_j$ if $i=j$, and 0 otherwise) and the graded Leibniz rule:
    \[ \partial (a_2 a_1) = \partial(a_2) a_1 + (-1)^{|a_2|} a_2 \partial(a_1). \]

    We define $\mathscr{A}_n$ to be the DG-subcategory of $\mathscr{I}_n$ with the same objects $X_1,\ldots, X_n$,
    but morphisms in $\bigoplus_{i,j} \mathrm{hom} (  X_i , X_j)$ are generated only by 
 \[ c^0_{ij} : 1 \leq i < j \leq n \ \mbox{ and } \  c^1_{ij} :  1\leq i,j \leq n. \] 

    As is customary, we use the same notation for a DG-category $\mathscr{C}$ and its total (endomorphism) algebra $\oplus_{X,Y \in Ob \mathscr{C}} \mathrm{hom} (X,Y)$.
    For example, $\mathscr{I}_n$ and $\mathscr{A}_n$ denote DG-algebras over $\k_n$. 
    On the other hand, given a ring map $\k \to \k_n$, $\mathscr{I}_n^\k$ and
    $\mathscr{A}_n^\k$ denote the corresponding DG-algebras over $\k$. 
The differential on $\mathscr{I}_n^\k$ (and similarly $\mathscr{A}_n^\k$) given by the same formulae as for $\mathscr{I}_n$ (and similarly for $\mathscr{A}_n)$ except that now some of the idempotents are identified and furthermore the algebra is generated by the given generators freely over $\k$ rather than $\k_n$. For example, in the case $\k = \K$ (as considered in \cite{EkNg}), all the $\delta_{ii}=1 \in \mathbb{K}$ and arbitrary words in $c_{ij}^p$ are allowed. We warn the reader that $\mathscr{I}_n^\mathbb{K}$ is quite different than considering $\mathscr{I}_n$ and viewing it as a DG-algebra over $\mathbb{K}$ by forgetting its $\k_n$-algebra structure.

\end{definition} 

A helpful way of thinking about these categories is to draw a circle with $n$ points ordered in counter-clockwise direction corresponding to the objects $X_i$ as in Figure
\ref{internalfig}. The set of generators $c^p_{ij}$ is in bijection with homotopy
classes of paths following a counter-clockwise flow that connects $i$ to $j$. The exponent $p$ counts the
number of times such paths pass from an additional marked point $*$ between $n$ and $1$.
Multiplication in the category can be thought of as broken paths, and the differential
corresponds to breaking the path into two paths of shorter length, with the
exception that in the case of $\d c_{ii}^1$, there is an additional idempotent term $e_i$ that appears in the
differential.

\begin{figure}[htb!]
\centering
\begin{tikzpicture}
\tikzset{vertex/.style = {style=circle,draw, fill,  minimum size = 2pt,inner        sep=1pt}}
\tikzset{edge/.style = {->,>=stealth',shorten >=8pt, shorten <=8pt  }}
\def \radius {1.5cm}
\def \margin {0} 

    \foreach \s in {1,2,3,4} {
\node[vertex] at ({360/5 * (\s)}:\radius) {} ;
  
    \draw ({360/5 * (\s)+\margin}:\radius) arc ({360/5 *(\s)+\margin}:{360/5*(\s+1)-\margin}:\radius);

}
  \draw ({0+\margin}:\radius) arc ({360/5 *0+\margin}:{360/5-\margin}:\radius);
\node at ({0}:\radius) {$*$} ;

\node[yshift=7, xshift=2] at ({360/5 }:\radius)  {$1$} ;
    \node[xshift=-5,yshift=2] at ({360/5 * (2)}:\radius)  {$2$} ;
    \node[xshift=-7, yshift=-4] at ({360/5 * (3)}:\radius)  {$\cdot$} ;
    \node[yshift=-9] at ({360/5 * (4)}:\radius)  {$n$} ;

    \draw[edge] ({360/5 +\margin}:\radius+3) arc ({360/5
    +\margin}:{360/5 *2-\margin}:\radius+3);
    \draw[edge] ({360/5 * (2)+\margin}:\radius+3) arc ({360/5
    *(2)+\margin}:{360/5*(3)-\margin}:\radius+3);
    \draw[edge] ({360/5 * (3)+\margin}:\radius+3) arc ({360/5
    *(3)+\margin}:{360/5*(4)-\margin}:\radius+3);
    \draw[edge] ({360/5 * (4)+\margin}:\radius+3) arc ({360/5
    *(4)+\margin}:{360/5*(6)-\margin}:\radius+3);

\end{tikzpicture}
    \caption{Visualisation of generators $c_{ij}^p$. }
	\label{internalfig}     
\end{figure}
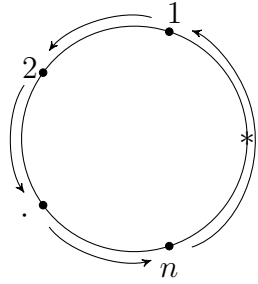

Next, let us introduce the following $A_\infty$-algebra, whose definition is attributed to Kontsevich in
\cite[Rmk. 3.11]{Se}.

\begin{definition} We define $\mathscr{K}_n$ to be the strictly unital $A_\infty$-category with
    $n$-objects $Z_1,\ldots, Z_{n}$ and morphisms given by, for $i \in \mathbb{Z}/n$,
    \begin{align*}
        \mathrm{hom}(Z_i,Z_i) &= \mathbb{K}e_i,  \\
        \mathrm{hom}(Z_i, Z_{i+1}) &= \mathbb{K} x_i, 
    \end{align*}
    where $e_i$ are idempotents. We again fix $n$ integers $(m_1,m_2,\ldots,m_n)$ and define gradings 
    \[ |x_i|= 1 + m_{i+1}-m_i \text{\ for\ } i=1,\ldots,n-1, \text{\ and \ } |x_{n}|= -1 +
    m_1 - m_n \]
The $A_\infty$-operations involving idempotents are determined by strict unitality. The only other
    products are given by:
    \[\mathfrak{m}_n(x_{i-1},\ldots , x_1 , x_0, x_{n-1},\ldots, x_{i+1}, x_i) = e_i, \text{\ \
    for \ \ } i \in \mathbb{Z}/n \]
\end{definition}

The fact that $A_\infty$-relations are satisfied is checked easily by the following equation that
holds for all $i
\in \mathbb{Z}/n$:
\begin{align*} &\mathfrak{m}_2(e_i, \mathfrak{m}_n(x_{i-1},\ldots , x_1 , x_0, x_{n-1},\ldots,
    x_{i+1}, x_i)) \\ &+ 
(-1)^{|e_i|-1} \mathfrak{m}_2( \mathfrak{m}_n(x_{i-1},\ldots , x_1 , x_0, x_{n-1},\ldots, x_{i+1},
x_i), e_i)  = \mathfrak{m}_2(e_i,x_i) - \mathfrak{m}_2(x_i,e_i) = 0 
\end{align*}

It can be shown that for $n \geq 3$,  $\mathscr{K}_n$ is not formal as an $A_\infty$-algebra, but we have 
\[ \mathscr{K}_2  \simeq \mathbf{k}_2 \langle x_1, x_2 \rangle / \langle x_2 x_1 =e_1, x_1 x_2 =
e_2 \rangle \ \  |x_1| = 1+ m_2-m_1, |x_2| = -1 + m_1 -m_2 \]

Recall that cobar and bar constructions extend to the setting of $A_\infty$-algebras. We denote these by
$\Omega$ and $\mathrm{B}$, respectively. (See \cite[Sec. 2]{EkLe} for a review that is compatible with our sign conventions.)

\begin{lem} $\mathscr{I}_n$ is isomorphic to $\Omega
    \mathrm{B}\mathscr{K}_n $. In particular, there exists a (strictly unital) $A_\infty$-functor
    \[ \mathfrak{f}: \mathscr{K}_n \to \mathscr{I}_n \]
    which is a quasi-isomorphism. 
\end{lem} 
\begin{proof} Recall that for any coaugmented conilpotent DG-colgebra $\mathscr{C}$ and an augmented DG-algebra $\mathscr{A}$ we have an adjunction  
\[ \mathrm{hom}_{DG} (\Omega \mathscr{C}, \mathscr{A}) \simeq    \mathrm{hom}_{coDG} (\mathscr{C}, \mathrm{B} \mathscr{A}) \]
(see \cite[Sec. 2.2.2]{EkLe} for a recent exposition of this well-known adjunction.) 
In particular, for any $A_\infty$-algebra $\mathscr{A}$, applying the adjunction  
\[ \mathrm{hom}_{DG} (\Omega \mathrm{B} \mathscr{A}, \Omega \mathrm{B} \mathscr{A}) \simeq \mathrm{hom}_{coDG} (\mathrm{B} \mathscr{A}, \mathrm{B} \Omega \mathrm{B} \mathscr{A} ) \]
and tracing the image of the identity map $\Omega \mathrm{B} \mathscr{A} \to \Omega \mathrm{B} \mathscr{A}$, we get a quasi-isomorphism of DG-colgebras $\mathrm{B} \mathscr{A} \to \mathrm{B} \Omega \mathrm{B} \mathscr{A}$, or equivalently an $A_\infty$-algebra quasi-isomorphism
\[ \mathfrak{f}: \mathscr{A} \to \Omega \mathrm{B} \mathscr{A} \]
for any $A_\infty$-algebra $\mathscr{A}$. This is sometimes referred to as the standard resolution of $A$ (see for ex. \cite[Sec. 16.7.3]{drinfeld}. 

Chasing through the adjunction (see \cite[Sec. 2.2.2]{EkLe}), one sees that the $A_\infty$ functor $\mathfrak{f} = (\mathfrak{f}^i)_{i \geq 1}$ has an explicit description: A composable word consisting of $i$ letters in $\mathscr{A}$ is sent by $\mathfrak{f}^i$ to the corresponding element of $B \mathscr{A}$ (whose elements by definition consist of words in $\mathscr{A}$). 

Here, we apply this construction to $\mathscr{A} = \mathscr{K}_n$. Indeed, the $A_\infty$-functor $\mathfrak{f}$
    is the Yoneda embedding defined by sending $\mathfrak{f}(Z_i) = X_{i}$ at the level of objects, and on
    morphisms, we set for $k \geq 1$:
    \[ \mathfrak{f} (x_{i+k-1},x_{i+k-2}, \ldots, x_{i} ) =
    (-1)^{|c_{ij}^p|} c_{ij}^p \]
    where $j$ and $p$ are determined uniquely by the equation $i+k = j + np$, $1 \leq j \leq n$.
    Here, the indices $i$ appearing on generators $x_i$ are considered in $\mathbb{Z}/n$.
    The $A_\infty$-functor equation takes the form (using the rules \cite[Eq. 4]{EL}):
    \[ 
    \sum_{0 \leq q \leq p} \sum_r (-1)^{|c_{ir}^{q} |} \mathfrak{f}(x_{j+np-1},\ldots, x_{r+nq})
    \mathfrak{f}(x_{r+nq-1}, \ldots, x_i)   + \partial c_{ij}^p = \delta_{ij}
    \delta_{1p}  \]  
which coincides precisely with the definition of the differential on $\mathscr{I}_n$. 
    
Note that in \cite[Sec.
    2]{EkLe}, the signs for the cobar-bar constructions are written out explicitly. It is an easy
    exercise to check that the signs are compatible with the signs that appear in Definition
    \ref{defint} of the internal algebra. 

Finally, the fact that $\mathfrak{f}$ is quasi-isomorphism follows from the bar-cobar adjunction as explained above.
\end{proof}

\begin{cor} \label{coh} $H^*(\mathscr{I}_n)$ is generated over $\mathbf{k}_n$ by the elements  
    \[ c_{12}^0, c_{23}^0, \ldots, c_{(n-1)n}^0, c_{n1}^{1}. \]
\end{cor}

This computation can be used to show that the DG-subalgebra $\mathscr{A}_n$ of $\mathscr{I}_n$
generated by only the elements  
\[ c^0_{ij} : 1 \leq i < j \leq n \ \mbox{ and } \  c^1_{ij} :  1\leq i,j \leq n \] 
is quasi-isomorphic to $\mathscr{I}_n$.
Here we provide a general proof over any base ring $\k$. 

\begin{thm} \label{retract} For every base ring $\k$, let $\mathfrak{i} : \mathscr{A}^\k_n \to \mathscr{I}^{\mathbf{k}}_n$ be the inclusion map. There exists a retraction $\mathfrak{r} : \mathscr{I}^\k_n \to \mathscr{A}^\k_n$ given by
    \begin{align*}
        \r_{|\mathscr{A}^\k_n} & = \mathrm{Id}_{\mathscr{A}^\k_n}\\
        \r (c_{ij}^{p}) &= C_{p-1} \sum_{k_1,\ldots,k_{2p-2}}  (-1)^{\dagger} c_{k_{2p-2} j}^1 \ldots c_{k_1 k_2}^1 c_{ik_1}^1 \ \ \text{for\ } p>1, \\
    \end{align*}
    where 
    $ C_{p}= \frac{1}{p+1}{2p \choose p}$ is the $p^{th}$ Catalan number and $\dagger = m_{k_1}+\ldots+m_{k_{2p-2}}$,
    that induces a quasi-isomorphism. The latter follows from the existence of a chain homotopy $\H$ between $\mathfrak{i} \circ \mathfrak{r} : \mathscr{I}^{\k}_n \to \mathscr{I}^{\k}_n$ and the identity map defined by 
\begin{align} \label{1}
        \H_{|\mathscr{A}^{\k}_n} & = 0\\
     \label{2}   \H (c_{ij}^{p}) &= \sum_{q=2}^{p} C_{p-q+1} \sum_{k_1,\ldots,k_{2p-2q+1}}  (-1)^{\ddagger}  c_{k_{2p-2q+1} j}^q  c_{k_{2p-2q} k_{2p-2q+1}}^1 \ldots c_{k_1k_2}^1c_{ik_1}^1 \ \ \text{for\ } p>1 \\
        \label{3} \H (a_2a_1) &= \H(a_2)\r(a_1)+(-1)^{|a_2|} a_2\H(a_1) 
    \end{align}
    where $\ddagger = m_{k_1}+m_{k_2}+\ldots+m_{k_{2p-2q+1}}+m_j$.

\end{thm}
\begin{proof} 
First, we check that $\r$ is a DG-algebra map, and as $\mathscr{I}^\k_n$ is free as an $\k$-algebra, this amounts to showing that $\r$ commutes with the differential on
    the generators.
   
   For $p>1$,
      \begin{align*}
  \r \d (c_{ij}^p) =&   \sum_{k_1,k_2,\ldots,k_{2p-1}}  (-1)^{\dagger+m_j+m_{k_{2p-1}}} C_{p-1} \left(c_{k_{2p-1} j}^0c_{k_{2p-2} k_{2p-1}}^1 \ldots c_{ik_1}^1 + c_{k_{2p-1} j}^1 \ldots c_{k_1 k_2}^1 c_{ik_1}^0\right)    \\
  +& \sum_{k_1,k_2,\ldots,k_{2p-3}}  (-1)^{\dagger+m_j-m_{k_{2p-2}}}(C_{p-2}+C_2C_{p-3} + \cdots C_{p-3}C_2+ C_{p-2}) c_{k_{2p-3} j}^1\ldots  c_{ik_1}^1
   \end{align*} 
   whereas 
  \begin{align*}
  \d \r (c_{ij}^p) = C_{p-1} & \left( \sum_{k_1,k_2,\ldots,k_{2p-1}}  (-1)^{\dagger+m_j+m_{k_{2p-1}}} \left(c_{k_{2p-1} j}^0c_{k_{2p-2} k_{2p-1}}^1 \ldots  c_{ik_1}^1 + c_{k_{2p-1} j}^1 \ldots c_{k_1 k_2}^1 c_{ik_1}^0\right)   \right. \\
  & \left. + \sum_{k_1,k_2,\ldots,k_{2p-3}}  (-1)^{\dagger+m_j-m_{k_{2p-2}}}c_{k_{2p-3}j}^1 \ldots c_{ik_1}^1\right)
   \end{align*}
   after cancellations. 
   The fact that Catalan numbers satisfy the recurrence relation
   $$C_{p-1} = \sum_{q=0}^{p-2} C_qC_{p-q-2}$$ 
   shows that $\r$ is a DG-algebra homomorphism.

Next, we prove that $\H$ is a chain homotopy between $\mathfrak{i} \circ \mathfrak{r}$ and the identity map. Observe that, by $(\ref{3})$, it suffices to check the condition $\partial \H + \H \partial = \mathfrak{i} \circ \r - \mathrm{Id}$ on the generating set $\{ c_{ij}^p \} $ which is clearly satisfied for $p=0,1$. From now on let $p \geq 2$.

In order to compute $(\partial \H + \H \partial) (c_{ij}^p)$, first we have 
$$\H \partial (c_{ij}^p) = \sum_{q=0}^p \sum_k (-1)^{m_k+m_j} \H (c_{kj}^q c_{ik}^{p-q})$$
and, by $(\ref{3})$, this is a sum of two expressions, first of which is
\begin{align} \label{4}
\sum_{q=2}^{p-1} &\sum_{r=2}^q C_{q-r+1} C_{p-q-1} \sum_{k_s} (-1)^{\epsilon} c_{k_{2p-2r}j}^r c_{k_{2p-2r-1}k_{2p-2r}}^1\cdots c_{ik_1}^1 \\
\label{5} +&\sum_{r=2}^p C_{p-r+1}  \sum_{k_s} (-1)^{\epsilon} c_{k_{2p-2r+2}j}^r c_{k_{2p-2r-1}k_{2p-2r}}^1\cdots c_{k_1k_2}^1c_{ik_1}^0 
\end{align}
and the second is 
\begin{align}
\label{6} \sum_{q=0}^p &\sum_{r=2}^{p-q} C_{p-q-r+1} \sum_{k_s} (-1)^{1+\epsilon} c_{k_{2p-2q-2r+2}j}^q c_{k_{2p-2q-2r+1}k_{2p-2q-2r+2}}^r c_{k_{2p-2q-2r}k_{2p-2q-2r+1}}^1\cdots c_{ik_1}^1 \\
\label{7} =\sum_{q'=2}^p &\sum_{r=2}^{q'} C_{p-q'+1} \sum_{k_s} (-1)^{1+\epsilon} c_{k_{2p-2q'+2}j}^{q'-r} c_{k_{2p-2q'+1}k_{2p-2q'+2}}^r c_{k_{2p-2q'}k_{2p-2q'+1}}^1\cdots c_{ik_1}^1
\end{align}
where $\epsilon=\sum_s m_{k_s}$. 
Note that the expression $(\ref{7})$ (which will be cancelled by a term below) is obtained from expression $(\ref{6})$ by a change of variable $q'=q+r$. 

We also have
$$ \partial \H (c_{ij}^p) = \sum_{q'=2}^p C_{p-q'+1} \sum_l (-1)^{\ddagger} \partial \left( c_{l_{2p-2q'+1} j}^{q'}  c_{l_{2p-2q'} l_{2p-2q'+1}}^1 \ldots c_{l_1l_2}^1c_{il_1}^1 \right) $$
and, by the graded Leibniz rule, this is a sum of two expressions, first of which is based on  $\left( \partial c_{l_{2p-2q'+1} j}^{q'} \right) \cdot  c_{l_{2p-2q'} l_{2p-2q'+1}}^1 \ldots c_{l_1l_2}^1c_{il_1}^1 $ 
\begin{align} 
\label{8} \sum_{q'=2}^{p} & C_{p-q'+1}\sum_{r=1}^{q'}  \sum_{k_s} (-1)^{\epsilon} c_{k_{2p-2q'+2}j}^{q'-r} c_{k_{2p-2q'+1}k_{2p-2q'+2}}^r c_{k_{2p-2q'}k_{2p-2q'+1}}^1\cdots c_{ik_1}^1 \\
\label{9} +\sum_{q'=2}^{p} & C_{p-q'+1} \ \  \sum_{k_s} (-1)^{\epsilon} c_{k_{2p-2q'+2}j}^{q'} c_{k_{2p-2q'+1}k_{2p-2q'+2}}^0 c_{k_{2p-2q'}k_{2p-2q'+1}}^1\cdots c_{ik_1}^1 
\end{align}
Note that $r \geq 2$ part of $(\ref{8})$ cancels $(\ref{7})$ above and $(\ref{9})$ is cancelled by a sum in the second part of $ \H \partial (c_{ij}^p)$ whose remaining contributions are as follows 
\begin{align} 
\label{10} \sum_{q'=2}^{p} & C_{p-q'+1}  \sum_{k_s} (-1)^{1+\epsilon} c_{k_{2p-2q'}j}^{q'} c_{k_{2p-2q'-1}k_{2p-2q'}}^1\cdots c_{ik_1}^1 \\
\label{11}+\sum_{q'=2}^p & C_{p-q'+1}  \sum_{k_s} (-1)^{1+\epsilon} c_{k_{2p-2q'+2}j}^r c_{k_{2p-2q'-1}k_{2p-2q'}}^1\cdots c_{k_1k_2}^1c_{ik_1}^0 
\end{align}
where $(\ref{10})$ is a result of the Kronecker delta component in the differential of $c_{kl}^1$ . 
Observe that $(\ref{11})$ is cancelled by $(\ref{5})$ and therefore $(\partial \H + \H \partial) (c_{ij}^p)$ is equal to the sum of the expressions $(\ref{4})$, $(\ref{10})$, and what remains from $(\ref{8})$ after cancellation by $(\ref{7})$, namely
$$\sum_{q'=2}^{p} C_{p-q'+1} \sum_{k_s} (-1)^{\epsilon} c_{k_{2p-2q'+2}j}^{q'-1} c_{k_{2p-2q'+1}k_{2p-2q'+2}}^1 \cdots c_{ik_1}^1 $$
whose $q'=2$ term is precisely $\r (c_{ij}^p)$.
Also note that $q'=p$ term of $(\ref{10})$ is exactly $c_{ij}^p$. 
Therefore $(\partial \H + \H \partial) (c_{ij}^p)$ is equal to $(\mathfrak{i} \circ \r - \mathrm{Id}) (c_{ij}^p)$ plus the following (which we obtain by switching the order of summation between $q$ and $r$ variables in $(\ref{4})$ and renaming the variables $q'$ in $(\ref{8})$ and $(\ref{10})$ as $r+1$ and $r$, respectively):
$$\sum_{r=2}^{p-1} \left( \left( \sum_{q=r}^{p-1}  C_{q-r+1} C_{p-q-1}\right) -C_{p-r+1} +C_{p-r}\right) \sum_{k_s} (-1)^{\epsilon} c_{k_{2p-2r}j}^r c_{k_{2p-2r-1}k_{2p-2r}}^1\cdots c_{ik_1}^1 $$
But the above expression vanishes since
$$C_{p-r+1}=\left( \sum_{q=r}^{p-1}  C_{q-r+1} C_{p-q-1}\right)  +C_{p-r}$$
by the recurrence relation satisfied by the Catalan numbers.  
\end{proof}


\begin{rmk}
For the readers who prefer a less computational approach, we outline another proof that uses more
    abstract machinery but is less explicit.

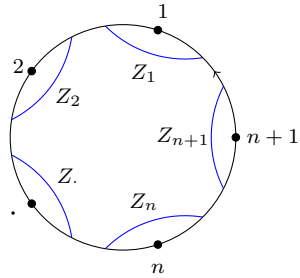
\begin{figure}[htb!]
\centering
\begin{tikzpicture}
\tikzset{vertex/.style = {style=circle,draw, fill,  minimum size = 2pt,inner        sep=1pt}}
\tikzset{edge/.style = {->,>=stealth',shorten >=8pt, shorten <=8pt  }}
\def \radius {1.5cm}
\def \margin {0} 
\tikzset{->-/.style={decoration={ markings,
        mark=at position #1 with {\arrow{>}}},postaction={decorate}}}

    \foreach \s in {1,2,3,4} {
\node[vertex] at ({360/5 * (\s)}:\radius) {} ;
  
    \draw ({360/5 * (\s)+\margin}:\radius) arc ({360/5 *(\s)+\margin}:{360/5*(\s+1)-\margin}:\radius);

\draw[blue] ({360/5 * (\s) - 360*3/40}:\radius) arc ({360/5
    *(\s)+360*3/40}:{360/5*(\s)-360*3/40}:-\radius);

}

\node[vertex] at ({360/5 * 5}:\radius) {} ;
 
\draw[->-=.5] ({0}:\radius) arc ({360/5 *0}:{360/5}:\radius);
\draw[blue] ({360/5 * (5) - 360*3/40}:\radius) arc ({360/5
    *(5)+360*3/40}:{360/5*(5)-360*3/40}:-\radius);

    \node[xshift=14] at ({0}:\radius) {\tiny $n+1$} ;

    \node[yshift=7, xshift=2] at ({360/5 }:\radius)  {\tiny $1$} ;
    \node[xshift=-5,yshift=2] at ({360/5 * (2)}:\radius)  {\tiny $2$} ;
    \node[xshift=-7, yshift=-4] at ({360/5 * (3)}:\radius)  {$\cdot$} ;
    \node[yshift=-9] at ({360/5 * (4)}:\radius)  {\tiny $n$} ;

 \node at ({0}:\radius-0.7cm) {\tiny $Z_{n+1}$};
 \node at ({360/5}:\radius-0.6cm) {\tiny $Z_{1}$};
   \node at ({360/5 *2}:\radius-0.6cm) {\tiny $Z_{2}$};
   \node at ({360/5 *3}:\radius-0.6cm) {\tiny $Z_{\cdot}$};
   \node at ({360/5 *4}:\radius-0.6cm) {\tiny $Z_{n}$};

\end{tikzpicture}
    \caption{$\mathscr{K}_{n+1}$ represented as the partially wrapped category of $\mathbb{D}^2$
    with $(n+1)$ marked points. }
	\label{pwrap}     
\end{figure}

    The category $\mathscr{K}_{n+1}$ appears as the partially wrapped Fukaya category of the
    $\mathbb{D}^2$ with $(n+1)$ marked points, see \cite[Sec. 3.3]{HKK}. In Figure \ref{pwrap}, one can find the
Lagrangians in this category that correspond to the objects $Z_i$ for $i=1,\ldots, n+1$. As noted in
    \cite[Sec. 3.3]{HKK}, one can easily check that the object $Z_{n+1}[-|x_{n+1}|]$ is isomorphic to
the twisted complex:
\[ Z_1 \xrightarrow{x_1} Z_2[|x_1|-1] \xrightarrow{x_2} \cdots \xrightarrow{x_{n-1}}
Z_{n}[|x_1|+|x_2|+ \ldots +|x_{n-1}|+1-n] \]
where the isomorphism maps are given by $x_{n}$ and $x_{n+1}$. 

    Thus, if we take the full subcategory of $\mathscr{K}_{n+1}$ consisting of objects $Z_1,\ldots,Z_{n}$
    and localise it along the above twisted complex representing $Z_{n+1}$, we should obtain $\mathscr{K}_{n}$ (as
    shown in \cite[Prop. 3.5]{HKK}, localising along the object $Z_{n+1}$ corresponds to the removal of
the marked point near which it is supported). 

Now, it is an easy exercise in localisation of DG-categories (\cite{drinfeld}, cf. \cite[Sec
3.5]{HKK}) to show that
localisation of $Z_1,\ldots,Z_{n}$ along the above twisted complex can be identified
    directly with $\mathscr{A}_{n}$ (cf. proof of Prop. \ref{drindrin} below). 
\end{rmk}

We note that cohomology of $\mathscr{I}_n$ and $\mathscr{I}_n^\k$ can be very
different. In what follows we will use the following result that follows from Theorem \ref{retract} and the Proposition \ref{drindrin} given below :
\begin{align} \mathscr{I}_2^\mathbb{K} \simeq \mathbb{K}[z,z^{-1}] \end{align}
In particular, note that this is infinite-dimensional over $\mathbb{K}$. On the other hand, as we proved in Corollary
\ref{coh}, cohomology of $\mathscr{I}_2$ is generated by two elements over $\k_2$. 

\begin{prop}
\label{drindrin}
$\mathscr{A}_2^\mathbb{K}$ is quasi-isomorphic to $\mathbb{K}[z,z^{-1}] $ considered with trivial differential. \end{prop}
\begin{proof}
To simplify the notation, let us denote $c_{12}^0, c_{21}^1, c_{11}^1, c_{22}^1$, and $c_{12}^1$ by $s, t, k, l$, and $u$, respectively. Thus, the algebra $\mathscr{A}_2^\mathbb{K}$ is the semifree DG-algebra, generated over $\mathbb{K}$ freely by $s,t,k,l,u$ with the differential determined by
\begin{align*}
dk &= 1-ts \\
dl &= 1-st \\
du &= ls-sk
\end{align*}
where the gradings are determined as follows $|s|=m, |t|=-m, |k|=|l|=-1, |u|=m-2$ for an arbitrary fixed integer $m$. 
We would like to show that this DG-algebra is quasi-isomorphic to the algebra $\mathbb{K}[z,z^{-1}]$ with $|z|=m$ considered as a DG-algebra with trivial differential.
We did not manage to construct a quasi-isomorphism explicitly even though, surely, such explicit construction is feasible. Instead, we will exploit the fact that the algebra $\mathscr{A}_2^\mathbb{K}$ is given by Drinfeld's construction of DG-quotient (as in \cite{drinfeld}).
Namely, let us consider the derived category $\mathcal{A}= D^b(\mathbb{K}[z])$. Let $L$ denote the object corresponding to $\mathbb{K}[z]$, and $S = Cone(L[-m] \xrightarrow{z} L) $, and $\mathcal{A}_{0}$ be the thick triangulated subcategory of $\mathcal{A}$ generated by $S$. Equivalently, $\mathcal{A}_{0}$ corresponds to modules of $\mathbb{K}[z]$ with support at 0. By the main result of \cite[Thm. 3.2]{miyachi} Verdier quotient $\mathcal{A}/ \mathcal{A}_0$ is identified with $D^b(\mathbb{K}[z,z^{-1}])$ and the endomorphism algebra of $L$ in this quotient is $End_{\mathcal{A}/\mathcal{A}_0}(L)=\mathbb{K}[z,z^{-1}]$. On the other hand, by viewing $\mathbb{K}[z]$ as DG-algebra with trivial differential, we can consider the DG-enhancement of $\mathcal{A}$, call it $\mathcal{A}^{pre}$. We can then apply Drinfeld's construction given in \cite[Sec. 3.1]{drinfeld} and take the DG quotient of $\mathcal{A}^{pre}$ by the full DG-subcategory $\mathcal{A}_0^{pre}$ consisting of $S$ (see \cite[Ex. 3.7.1]{drinfeld} for a similar looking example corresponding to $\mathscr{A}_2^{\k_2}$). The endomorphism algebra of $L$ in $\mathcal{A}^{pre}/\mathcal{A}_0^{pre}$ is generated by $s,t,k,l,u$ where these generators can be identified with the morphisms of twisted complexes given in Figure \ref{dgquotient}. Moreover, the differential of these generators are easily computed to give an identification of $end_{\mathcal{A}^{pre}/\mathcal{A}_0^{pre}}(L)$ with $\mathscr{A}_2^\mathbb{K}$. Now, by the general theorem of Drinfeld \cite[Thm. 3.4]{drinfeld}, the homotopy categories of these two different quotients are isomorphic. In particular, \[ H^*(end_{\mathcal{A}^{pre}/\mathcal{A}_0^{pre}}(L)) \simeq End_{\mathcal{A}/\mathcal{A}_0}(L). \] It follows that
\[ H^*\mathscr{A}_2^{\mathbb{K}} = \mathbb{K}[z,z^{-1}] \]
Finally, it is well known that $\mathbb{K}[z,z^{-1}]$ is intrinsically formal for any value of $m$. Since one can easily compute $HH^*(\mathbb{K}[z,z^{-1}]) = \mathbb{K}[z,z^{-1},\partial_z]$ and check that $HH^2_{<0}(\mathbb{K}[z,z^{-1}])=0$. Therefore, it follows that $\mathscr{A}_2^{\mathbb{K}}$ is quasi-isomorphic to $\mathbb{K}[z,z^{-1}]$.  
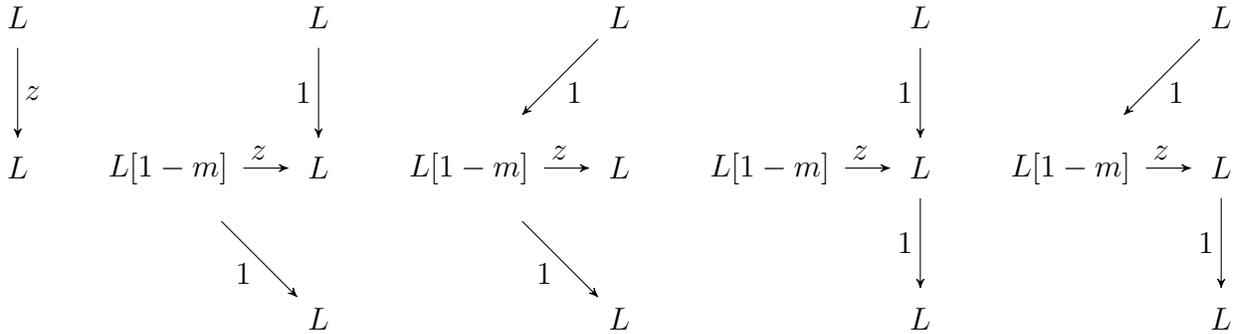
\begin{figure}[h!]
    \centering
\tikzstyle{level 1}=[level distance=3cm, sibling distance=1.8cm]
\tikzstyle{level 2}=[level distance=3cm, sibling distance=0.8cm]
\tikzstyle{level 3}=[level distance=3cm, sibling distance=0.8cm]

\tikzstyle{bag} = [circle, minimum width=3pt,fill, inner sep=0pt]
\tikzstyle{end} = [circle, minimum width=3pt,fill, inner sep=0pt]
    \tikzset{edge/.style = {->,>=stealth' }}

    \begin{tikzpicture} [scale=1,auto=left,every node/.style={circle}]

\node (n1) at (0,2) {$L$};
\node (n2) at (0,4) {$L$}; 
\node (n3) at (0.2,3) {$z$};

\node (n4) at (4,4) {$L$};
\node (n5) at (4,2) {$L$}; 
\node (n6) at (2,2) {$L[1-m]$};
\node (n7) at (4,0) {$L$};
\node at (3.2, 2.2) {$z$};
\node at (3.8,3) {$1$};
\node at (3,0.6) {$1$};

\node (n8) at (8,4) {$L$};
\node (n9) at (8,2) {$L$}; 
\node (n10) at (6,2) {$L[1-m]$};
\node (n11) at (8,0) {$L$};
\node at (7.2, 2.2) {$z$};
\node at (7.4,3) {$1$};
\node at (7,0.6) {$1$};

\node (n12) at (12,4) {$L$};
\node (n13) at (12,2) {$L$}; 
\node (n14) at (10,2) {$L[1-m]$};
\node (n15) at (12,0) {$L$};
\node at (11.2, 2.2) {$z$};
\node at (11.8,3) {$1$};
\node at (11.8,1) {$1$};

\node (n16) at (16,4) {$L$};
\node (n17) at (16,2) {$L$}; 
\node (n18) at (14,2) {$L[1-m]$};
\node (n19) at (16,0) {$L$};
\node at (15.2, 2.2) {$z$};
\node at (15.4,3) {$1$};
\node at (15.8,1) {$1$};

\draw[edge] (n2) -- (n1)  ;

\draw[edge] (n4) -- (n5)  ;
\draw[edge] (n6) -- (n5)  ;
\draw[edge] (n6) -- (n7)  ;

\draw[edge] (n8) -- (n10)  ;
\draw[edge] (n10) -- (n9)  ;
\draw[edge] (n10) -- (n11)  ;

\draw[edge] (n12) -- (n13)  ;
\draw[edge] (n14) -- (n13)  ;
\draw[edge] (n13) -- (n15)  ;

\draw[edge] (n16) -- (n18)  ;
\draw[edge] (n18) -- (n17)  ;
\draw[edge] (n17) -- (n19)  ;

\end{tikzpicture}
    \caption{Primitive generators in the Drinfeld quotient, identified with $s,t,k,l,u$ from left-to-right}
    \label{dgquotient} 
\end{figure}

\end{proof}

\begin{rmk} Another proof of the Proposition \ref{drindrin} can be given by using partially wrapped Fukaya categories. Namely, consider a cylinder $T^*S^1$ with a single stop along one of its boundary components. The cotangent fiber $L$ generates this category and one has $End(L) = \mathbb{K}[z]$. On the other hand, the object supported near the stop is quasi-isomorphic to $Cone(L \xrightarrow{z} L)$, localizing with respect to that object corresponds to removing the stop (by \cite[Prop. 3.5]{HKK}). Thus, we can identify the endomorphism algebra of $L$ in the fully wrapped Fukaya category of $T^*S^1$ as this localisation which (as we showed in the proof of Proposition \ref{drindrin}) coincides with $\mathscr{A}_2^\mathbb{K}$. Now, one can directly compute the endomorphism algebra of $L$ in the fully wrapped Fukaya category of $T^*S^1$, which gives $\mathbb{K}[z,z^{-1}]$.
\end{rmk}

\section{Plumbings of cotangent bundles of 2-spheres}
\label{g=0}

Throughout this section, we assume that $\Gamma$ is a finite graph with $\mathbf{g}=0$. 
We first construct a Legendrian link $L_\Gamma$ in $\#^k(S^1 \times S^2, \xi_{std})$ associated to $\Gamma$.  
An important feature of this construction is that the symplectic $4$-manifold obtained by adding
symplectic $2$-handles along $L_\Gamma$ to $\#^k(S^1 \times \mathbb{D}^3, \omega_{std})$ is
homotopic, as a Weinstein manifold, to the plumbing of copies of $T^*S^2$ with respect to the graph $\Gamma$. 

After the construction, we compute the Chekanov-Eliashberg DG-algebra $\mathscr{B}_\Gamma = CE^*(L_\Gamma)$ of $L_\Gamma$ and prove that it is quasi-isomorphic to the derived multiplicative preprojective algebra associated to $\Gamma$.

\subsection{A Legendrian surgery presentation of plumbings}\label{resolution}

Choose a spanning tree $T$ of $\Gamma$ and embed it in the upper half plane $\mathbb{R}
\times \mathbb{R}_{\geq 0}$ such that all the vertices  lie on $\mathbb{R}
\times  \{0\}$. 
This can be easily done recursively (in multiple ways). 
We give an illustrating picture in Fig.~(\ref{fig4}) and leave the elaboration of the recursive argument to the reader.

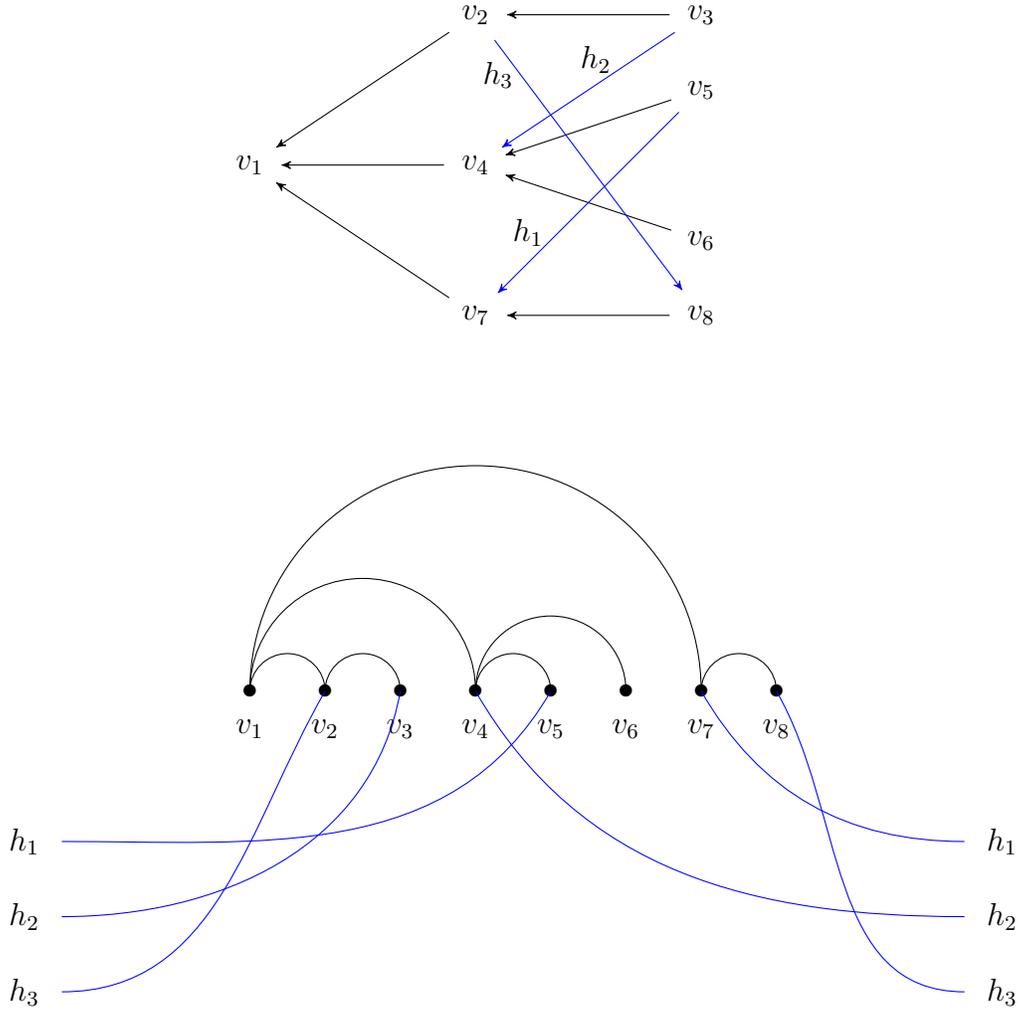
\begin{figure}[htb!]
	\centering

\tikzstyle{level 1}=[level distance=3cm, sibling distance=1.8cm]
\tikzstyle{level 2}=[level distance=3cm, sibling distance=0.8cm]
\tikzstyle{level 3}=[level distance=3cm, sibling distance=0.8cm]

\tikzstyle{bag} = [circle, minimum width=3pt,fill, inner sep=0pt]
\tikzstyle{end} = [circle, minimum width=3pt,fill, inner sep=0pt]
    \tikzset{edge/.style = {->,>=stealth' }}

    \begin{tikzpicture} [scale=1,auto=left,every node/.style={circle}]

    \node (n1) at (0,0) {$v_1$};
       \node (n2) at (3,2)  {$v_2$};
  \node (n5) at (3,0)  {$v_4$};
  \node (n9) at (3,-2) {$v_7$};
  \node (n4) at (6,2)  {$v_3$};
  \node (n7) at (6,1)  {$v_5$};
  \node (n8) at (6,-1)  {$v_6$};
 \node (n10) at (6,-2)  {$v_8$};

  \foreach \to/\from in {n1/n2,n1/n5,n1/n9,n2/n4,n5/n7,n5/n8,n9/n10}
    \draw[edge] (\from) -- (\to);
    
  \draw[blue][edge] (n2) -- (n10); 
  \draw[blue][edge] (n4) -- (n5); 
   \draw[blue][edge] (n7) -- (n9); 

 \node at (3.7, -0.9) {$h_1$}; 
  \node at (4.6, 1.4) {$h_2$}; 
   \node at (3.3,1.2) {$h_3$}; 
       
        \begin{scope}[shift={(0,-7)}] 
        \node at (0,0) {$\bullet$};
        \node at (1,0) {$\bullet$};
        \node at (2,0) {$\bullet$};
        \node at (3,0) {$\bullet$};
        \node at (4,0) {$\bullet$};
        \node at (5,0) {$\bullet$};
        \node at (6,0) {$\bullet$};
        \node at (7,0) {$\bullet$};

        \node at (0,-0.5) {$v_1$};
        \node at (1,-0.5) {$v_2$};
            \node at (2,-0.5) {$v_3$};
        \node at (3,-0.5) {$v_4$};
        \node at (4,-0.5) {$v_5$};
        \node at (5,-0.5) {$v_6$};
        \node at (6,-0.5) {$v_7$};
        \node at(7,-0.5)  {$v_8$};
 
        \draw (1,0) arc(0:180:0.5);
        \draw (2,0) arc(0:180:0.5);
        \draw (3,0) arc(0:180:1.5);
        \draw (4,0) arc(0:180:0.5);
        \draw (5,0) arc(0:180:1);
        \draw (6,0) arc(0:180:3);
        \draw (7,0) arc(0:180:0.5);
      
        \node at (-3,-2) {$h_1$};
        \node at (-3,-3) {$h_2$};
        \node at (-3,-4) {$h_3$};
        \node at (10,-2) {$h_1$};
        \node at (10,-3) {$h_2$};
        \node at (10,-4) {$h_3$};

        \draw [blue] (-2.5,-2) to[in=240,out=0] (4,0); 
        \draw [blue] (9.5,-2) to[in=300,out=180] (6,0); 
        \draw [blue] (-2.5,-3)[in=260,out=0] to (2,0); 
        \draw [blue] (9.5,-3)[in=300,out=180] to (3,0); 
        \draw [blue] (-2.5,-4)[in=240,out=0] to (1,0); 
            \draw [blue] (9.5,-4)[in=300,out=180] to (7,0); 

         \end{scope}

    \end{tikzpicture}
	\caption{The graph $\Gamma$ in a standard form with respect to a chosen orientation}
	\label{fig4}     
\end{figure}

Corresponding to the spanning tree $T$, we will have 2-handles in the Legendrian surgery picture of $X_\Gamma$. 
The complement of $T$ will correspond to 1-handles. 

Before we put the graph $\Gamma$ in a standard form, we choose an orientation and order on the edges of $\Gamma$. 
Note that, as is the case for the multiplicative preprojective algebra $\Lambda_\Gamma$, $\mathscr{B}_\Gamma$ is also independent of these choices.
Then for each arrow $h \in \Gamma \setminus T$ we introduce a pair of arcs in the
lower half plane $\mathbb{R} \times \mathbb{R}_{< 0}$ as in Fig.~(\ref{fig4}). More precisely, we fix $m, M$ as real numbers smaller and larger than the $x$-coordinates of $v_1$ and $v_s$, respectively, and if $h$ is an arrow from $v$ to $w$ which is the $i^{th}$ arrow in $\Gamma \setminus T$, then there will be a curve from $v$ to the point $(m, -i)$ with nonnegative slope throughout and another curve from $w$ to the point $(M, -i)$ with nonpositive slope, chosen so that there are no triple intersections.
We assume that the arrows in $T$ are oriented from right to left and also assume that in the ordering of all the edges of $\Gamma$, the edges in $T$ are ahead of the edges in $\Gamma \setminus T$.

Once $\Gamma$ is in a standard form, the endpoints of each pair of arcs corresponding to the same arrow in $\Gamma \setminus T$ is replaced by the feet, i.e. the attaching sphere $S^0 \times D^3$, of a $1$-handle and each vertex $v$ is replaced by the standard Lagrangian projection of a Legendrian unknot $U_v$ with $tb=-1$.
For each arrow $a$ from $v$ to $w$, we introduce an arc on the left half of $U_v$ and a meridian $m_a$ of the right half of $U_w$, i.e. the boundary of a small disk intersecting it transversely once, (so that all meridians and arcs on an unknot are ordered according to the chosen order of the edges), 
replace $a$ by a band from the arc on $U_v$ to an arc on $m_a$, and then connect sum $U_v$ and $m_a$ using this band. 
None of these bands contain a twist other than the one near the left foot of the corresponding 1-handle (dictated by the resolution operation \cite[Def. 2.3.]{EkNg}).
There is no difference between the treatment of arrows in $T$ and the rest, except that for $a \in T$, the corresponding band is planar whereas for $a \notin T$  the band goes over the associated $1$-handle and it may go over/under other such bands.
Note that these connected sum operations are performed directly on the Lagrangian projection without reverting back to the front projection. To justify this, observe that one could use the front projection and perform the connected sum there then one would end up with a more complicated diagram similar to the diagram given in the middle part of Figure \ref{fig0}. However, as in that example, there is a simplifying Reidemeister III move that gives the bottom part of Figure \ref{fig0}. This is a local move and we apply this move in all our connected sums right away. Figures \ref{fig4}, \ref{fig5} and \ref{fig6} describe how the Legendrian $L_\Gamma$ is constructed
in a sufficiently complicated example, from which the general pattern is clear. 

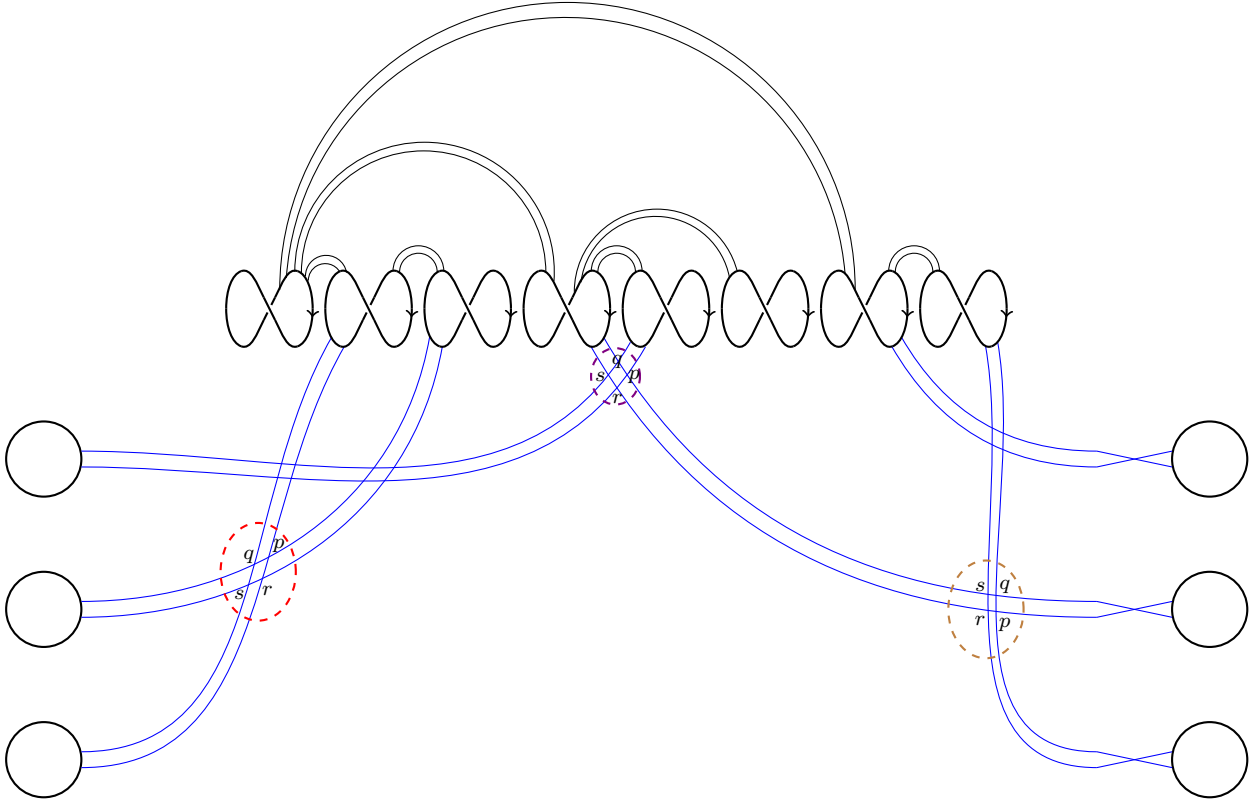
\begin{figure}[h!]
    \centering
    \begin{tikzpicture}

    \tikzset{->-/.style={decoration={ markings,
                mark=at position #1 with {\arrow{>}}},postaction={decorate}}}

        \begin{scope}[scale=0.25]

    \draw (8.1,1.93) arc(10:180:1.1);
    \draw [xshift=-10] (8.1,1.95) arc(30:180:0.9);

 \draw [xshift=150] (8,2) arc(0:180:1.35);
    \draw [xshift=140] (8,1.9) arc(-3:180:1);
    
 \draw [xshift=450] (8,2) arc(0:180:1.35);
    \draw [xshift=440] (8,1.9) arc(-3:180:1);
    
    \draw [xshift=900] (8,2) arc(0:180:1.35);
    \draw [xshift=890] (8,1.9) arc(-3:180:1);
       
    \draw [xshift=-10] (5.7,1.96) arc(180:-4:6.9);
    \draw (5.7,1.9) arc(180:0:6.5);
      
    \draw [xshift=-10] (4.9,1) arc(180:0:15.3);
    \draw  (4.9,1.64) arc(176:5:14.9);

    \draw [xshift=430] (5.1,0.9) arc(180:14:4.4);
    \draw [xshift=440] (5.1,1.4) arc(173:11:4);

        \end{scope}

 \draw [thick] (-2,-2) circle (0.5); 
    \draw [thick] (-2,-4) circle (0.5); 
    \draw [thick] (-2,-6) circle (0.5); 

    \draw [thick] (13.5,-2) circle (0.5); 
    \draw [thick] (13.5,-4) circle (0.5); 
    \draw [thick] (13.5,-6) circle (0.5); 

    \draw [blue] [yshift=-3](-1.5,-2) to[in=240,out=0] (6.0,-0.4); 
    \draw [blue] [yshift=3](-1.5,-2) to[in=240,out=0] (5.8,-0.55); 
    \draw [blue] [yshift=-3](-1.5,-4)[in=260,out=0] to (3.3,-0.39); 
    \draw [blue] [yshift=3](-1.5,-4)[in=260,out=0] to (3.13,-0.48);  
    \draw [blue] [yshift=-3](-1.5,-6)[in=240,out=0] to (2,-0.39); 
    \draw [blue] [yshift=3](-1.5,-6)[in=240,out=0] to (1.83,-0.48);

        \draw [blue][yshift=-3](12,-2) to[in=300,out=180] (9.27,-0.39); 
        \draw [blue][yshift=3](12,-2) to[in=300,out=180] (9.4,-0.48); 
    
        \draw [blue]  [yshift=3](12,-2) to [yshift=-6](13,-2);
        \draw [blue]  [yshift=-3](12,-2) to [yshift=6](13,-2);

        \draw [blue][yshift=-3] (12,-4)[in=300,out=180] to (5.27,-0.39); 
        \draw [blue][yshift=3] (12,-4) to[in=300,out=180] (5.44,-0.48); 

  \draw [blue]  [yshift=3](12,-4) to [yshift=-6](13,-4);
        \draw [blue]  [yshift=-3](12,-4) to [yshift=6](13,-4);

        \draw [blue][yshift=-3] (12,-6)[in=280,out=180] to (10.52,-0.39); 
        \draw [blue][yshift=3] (12,-6) to[in=280,out=180] (10.68,-0.54); 
  

   \draw [blue]  [yshift=3](12,-6) to [yshift=-6](13,-6);
        \draw [blue]  [yshift=-3](12,-6) to [yshift=6](13,-6);

        \begin{scope}[scale=0.25, yscale=2]

        \begin{scope}[xshift=0]

    \draw [thick=1.5] (2.5,1) to[in=90,out=190] (1.7,0);

   \draw [thick=1.5] (2.5,-1) to[in=270,out=170] (1.7,0);

    \draw [thick=1.5] (2.5,1) to[in=135,out=10] (3.7,0.3);
    \draw [thick=1.5] (2.5,-1)to[in=225,out=350] (3.7,-0.3)  ;

    \draw [thick=1.5] (5.5,1) to[in=90,out=350] (6.3,0);
    \draw [thick=1.5] (5.5,1) to[in=45,out=170] (4.3,0.3);
    \draw [thick=1.5, ->-=.2]  (6.3,0)to[in=10,out=270] (5.5,-1);
    \draw [thick=1.5] (5.5,-1)to[in=315,out=190] (4.3,-0.3)  ;

    \draw [thick=1.5] (3.7,0.3) to (4.3,-0.3);
    \draw [thick=1.5] (3.7,-0.3) to (3.9,-0.1);
    \draw [thick=1.5] (4.1,0.1) to (4.3,0.3);

    \end{scope}

      \begin{scope}[xshift=150]

    \draw [thick=1.5] (2.5,1) to[in=90,out=190] (1.7,0);

   \draw [thick=1.5] (2.5,-1) to[in=270,out=170] (1.7,0);

    \draw [thick=1.5] (2.5,1) to[in=135,out=10] (3.7,0.3);
    \draw [thick=1.5] (2.5,-1)to[in=225,out=350] (3.7,-0.3)  ;

    \draw [thick=1.5] (5.5,1) to[in=90,out=350] (6.3,0);
    \draw [thick=1.5] (5.5,1) to[in=45,out=170] (4.3,0.3);
    \draw [thick=1.5, ->-=.2]  (6.3,0)to[in=10,out=270] (5.5,-1);
    \draw [thick=1.5] (5.5,-1)to[in=315,out=190] (4.3,-0.3)  ;

    \draw [thick=1.5] (3.7,0.3) to (4.3,-0.3);
    \draw [thick=1.5] (3.7,-0.3) to (3.9,-0.1);
    \draw [thick=1.5] (4.1,0.1) to (4.3,0.3);

    \end{scope} 

 \begin{scope}[xshift=300]

    \draw [thick=1.5] (2.5,1) to[in=90,out=190] (1.7,0);

   \draw [thick=1.5] (2.5,-1) to[in=270,out=170] (1.7,0);

    \draw [thick=1.5] (2.5,1) to[in=135,out=10] (3.7,0.3);
    \draw [thick=1.5] (2.5,-1)to[in=225,out=350] (3.7,-0.3)  ;

    \draw [thick=1.5] (5.5,1) to[in=90,out=350] (6.3,0);
    \draw [thick=1.5] (5.5,1) to[in=45,out=170] (4.3,0.3);
    \draw [thick=1.5, ->-=.2]  (6.3,0)to[in=10,out=270] (5.5,-1);
    \draw [thick=1.5] (5.5,-1)to[in=315,out=190] (4.3,-0.3)  ;

    \draw [thick=1.5] (3.7,0.3) to (4.3,-0.3);
    \draw [thick=1.5] (3.7,-0.3) to (3.9,-0.1);
    \draw [thick=1.5] (4.1,0.1) to (4.3,0.3);

    \end{scope}

      \begin{scope}[xshift=450]

    \draw [thick=1.5] (2.5,1) to[in=90,out=190] (1.7,0);

   \draw [thick=1.5] (2.5,-1) to[in=270,out=170] (1.7,0);

    \draw [thick=1.5] (2.5,1) to[in=135,out=10] (3.7,0.3);
    \draw [thick=1.5] (2.5,-1)to[in=225,out=350] (3.7,-0.3)  ;

    \draw [thick=1.5] (5.5,1) to[in=90,out=350] (6.3,0);
    \draw [thick=1.5] (5.5,1) to[in=45,out=170] (4.3,0.3);
    \draw [thick=1.5, ->-=.2]  (6.3,0)to[in=10,out=270] (5.5,-1);
    \draw [thick=1.5] (5.5,-1)to[in=315,out=190] (4.3,-0.3)  ;

    \draw [thick=1.5] (3.7,0.3) to (4.3,-0.3);
    \draw [thick=1.5] (3.7,-0.3) to (3.9,-0.1);
    \draw [thick=1.5] (4.1,0.1) to (4.3,0.3);

    \end{scope} 

 \begin{scope}[xshift=600]

    \draw [thick=1.5] (2.5,1) to[in=90,out=190] (1.7,0);

   \draw [thick=1.5] (2.5,-1) to[in=270,out=170] (1.7,0);

    \draw [thick=1.5] (2.5,1) to[in=135,out=10] (3.7,0.3);
    \draw [thick=1.5] (2.5,-1)to[in=225,out=350] (3.7,-0.3)  ;

    \draw [thick=1.5] (5.5,1) to[in=90,out=350] (6.3,0);
    \draw [thick=1.5] (5.5,1) to[in=45,out=170] (4.3,0.3);
    \draw [thick=1.5, ->-=.2]  (6.3,0)to[in=10,out=270] (5.5,-1);
    \draw [thick=1.5] (5.5,-1)to[in=315,out=190] (4.3,-0.3)  ;

    \draw [thick=1.5] (3.7,0.3) to (4.3,-0.3);
    \draw [thick=1.5] (3.7,-0.3) to (3.9,-0.1);
    \draw [thick=1.5] (4.1,0.1) to (4.3,0.3);

    \end{scope}

      \begin{scope}[xshift=750]

    \draw [thick=1.5] (2.5,1) to[in=90,out=190] (1.7,0);

   \draw [thick=1.5] (2.5,-1) to[in=270,out=170] (1.7,0);

    \draw [thick=1.5] (2.5,1) to[in=135,out=10] (3.7,0.3);
    \draw [thick=1.5] (2.5,-1)to[in=225,out=350] (3.7,-0.3)  ;

    \draw [thick=1.5] (5.5,1) to[in=90,out=350] (6.3,0);
    \draw [thick=1.5] (5.5,1) to[in=45,out=170] (4.3,0.3);
    \draw [thick=1.5, ->-=.2]  (6.3,0)to[in=10,out=270] (5.5,-1);
    \draw [thick=1.5] (5.5,-1)to[in=315,out=190] (4.3,-0.3)  ;

    \draw [thick=1.5] (3.7,0.3) to (4.3,-0.3);
    \draw [thick=1.5] (3.7,-0.3) to (3.9,-0.1);
    \draw [thick=1.5] (4.1,0.1) to (4.3,0.3);

    \end{scope}

 \begin{scope}[xshift=900]

    \draw [thick=1.5] (2.5,1) to[in=90,out=190] (1.7,0);

   \draw [thick=1.5] (2.5,-1) to[in=270,out=170] (1.7,0);

    \draw [thick=1.5] (2.5,1) to[in=135,out=10] (3.7,0.3);
    \draw [thick=1.5] (2.5,-1)to[in=225,out=350] (3.7,-0.3)  ;

    \draw [thick=1.5] (5.5,1) to[in=90,out=350] (6.3,0);
    \draw [thick=1.5] (5.5,1) to[in=45,out=170] (4.3,0.3);
    \draw [thick=1.5, ->-=.2]  (6.3,0)to[in=10,out=270] (5.5,-1);
    \draw [thick=1.5] (5.5,-1)to[in=315,out=190] (4.3,-0.3)  ;

    \draw [thick=1.5] (3.7,0.3) to (4.3,-0.3);
    \draw [thick=1.5] (3.7,-0.3) to (3.9,-0.1);
    \draw [thick=1.5] (4.1,0.1) to (4.3,0.3);

    \end{scope}

      \begin{scope}[xshift=1050]

    \draw [thick=1.5] (2.5,1) to[in=90,out=190] (1.7,0);

   \draw [thick=1.5] (2.5,-1) to[in=270,out=170] (1.7,0);

    \draw [thick=1.5] (2.5,1) to[in=135,out=10] (3.7,0.3);
    \draw [thick=1.5] (2.5,-1)to[in=225,out=350] (3.7,-0.3)  ;

    \draw [thick=1.5] (5.5,1) to[in=90,out=350] (6.3,0);
    \draw [thick=1.5] (5.5,1) to[in=45,out=170] (4.3,0.3);
    \draw [thick=1.5, ->-=.2]  (6.3,0)to[in=10,out=270] (5.5,-1);
    \draw [thick=1.5] (5.5,-1)to[in=315,out=190] (4.3,-0.3)  ;

    \draw [thick=1.5] (3.7,0.3) to (4.3,-0.3);
    \draw [thick=1.5] (3.7,-0.3) to (3.9,-0.1);
    \draw [thick=1.5] (4.1,0.1) to (4.3,0.3);

\draw[red,thick,dashed] (-33.5,-7.0) ellipse (2.0cm and 1.3cm);
\draw[violet,thick,dashed] (-14.5,-1.8) ellipse (1.3cm and 0.75cm);
\draw[brown,thick,dashed] (5.2,-8.0) ellipse (2.0cm and 1.3cm);

\node at (-32.4, -6.3) {\tiny $p$};
\node at (-33, -7.5) {\tiny $r$};
\node at (-34, -6.6) {\tiny $q$};
\node at (-34.5, -7.6) {\tiny $s$};

\node at (-13.5, -1.8) {\tiny $p$};
\node at (-14.4, -2.4) {\tiny $r$};
\node at (-14.4, -1.4) {\tiny $q$};
\node at (-15.3, -1.8) {\tiny $s$};

\node at (6.2, -8.4) {\tiny $p$};
\node at (4.9, -8.3) {\tiny $r$};
\node at (6.2, -7.4) {\tiny $q$};
\node at (4.9, -7.4) {\tiny $s$};

    \end{scope} 

    \end{scope}

\end{tikzpicture}
    \caption{The configuration of bands and different types of secondary crossings. At the intersection of two bands, the one induced by the arc with the smaller slope in Figure \ref{fig4} goes over the other. For a more detailed local picture around an unknot see Figure~\ref{fig6}.} 
    \label{fig5}

\end{figure}

Now that we have a standard Lagrangian projection of $L_\Gamma$, here are the crossings:
\begin{itemize}
\item a pair of crossings $c_a$ and $c_a^*$ between $U_{s(a)}$ and $U_{t(a)}$ for each arrow $a$ of $\Gamma$
	\item a self-crossing $\tau_v$ of $U_v$ for each vertex $v$ of $\Gamma$
	\item a self-crossing $\zeta_a$ of $U_{s(a)}$ near the right foot of the $1$-handle $h_a$, for
        each arrow $a$ in $\Gamma \backslash T$.
	\item additional crossings associated to nonplanar bands in the above construction of $L_\Gamma$ which go over/under each other 
\end{itemize}


\subsection{Computation of $\mathscr{B}_\Gamma = CE^*(L_\Gamma)$}\label{CE-comp}

The computations in this section are based on the simplification of the Chekanov-Eliashberg DG-algebra of a Legendrian link in $\#^k (S^1 \times S^2, \xi_{std})$ presented in Sec.~(\ref{simplify}), where the internal DG-algebra $\mathscr{I}_n$ is proved to be quasi-isomorphic to its finitely generated DG-subalgebra $\mathscr{A}_n$ as well as the $A_\infty$-algebra $\mathscr{K}_n$. 

The resolution of $L_\Gamma$ to be used is the one constructed in Sec.~(\ref{resolution}) with the
choice of orientation and base-points indicated in Fig.~(\ref{fig6}).
In this resolution, for every $1$-handle $h_a$ induced by an arrow $a \in \Gamma \setminus T$, there are exactly two strands going over $h_a$ which belong to the same component of $L_\Gamma$. Therefore the internal generators $c^0_{12;a}, c^1_{21;a} $ can be replaced by the generators of $\mathscr{I}_2^\K \simeq \mathbb{K}[z_a,z_a^{-1}]$  with grading $0$ and the rest of the internal generators associated to $h_a$ can be dropped from the presentation.

Except for $t_v^\pm$ and $z_a^\pm$, all the generators of the Chekanov-Eliashberg algebra $\mathscr{B}_\Gamma$ come from the crossings in the resolution. 
We use the same notation for these generators as the corresponding crossings, and list them as
follows (see Fig.~(\ref{fig6})). 
	\begin{figure}[htb!]
		\centering
		\begin{tikzpicture}
	    \tikzset{->-/.style={decoration={ markings,
                mark=at position #1 with {\arrow{>}}},postaction={decorate}}}

   \draw [thick] (-5,-4) circle (0.5); 
    \draw [thick] (6,-5) circle (0.5); 
 
            \begin{scope}[xshift=-6cm]
    \draw [thick=1.5] (1.5,1) to[in=45,out=170] (0.3,0.3);
    \draw [thick=1.5]  (1.5,-1)to[in=270,out=10] (2.3,0);
    \draw [thick=1.5] (1.5,1) to[in=140,out=350] (1.9,0.8);
    \draw [thick=1.5] (2.05,0.65) to[in=90,out=320] (2.3,0); 
    \draw [thick=1.5] (1.5,-1)to[in=315,out=190] (0.3,-0.3)  ;
            
    \draw [thick=1.5] (1.95,0.7) to[in=90, out=230] (1.55,0.2);        
     \draw [thick=1.5] (1.55,0.2) to[in=230, out=270] (2.1,0.5);        
           
    \draw[thick=1.5] (1.95,0.7) to[in=160, out=50] (4.6cm,1);             
    \draw[thick=1.5] (2.2,0.6) to[in=160, out=50] (4.4cm,0.92);        
    
    \node at (2.6,0.6) {\tiny $c^*_{a_1}$}; 
    \node at (1.9,1.2) {\tiny $c_{a_1}$};

            \end{scope}

    \begin{scope}[xshift=6cm]
   \draw [thick=1.5, ->-=.7] (-1.6,0.93) to[in=90,out=190] (-2.3,0);
    \draw [thick=1.5] (-1.4,1) to[in=135,out=10] (-0.3,0.3);
    \draw [thick=1.5] (-1.4,-1)to[in=225,out=350] (-0.05,-0.05)  ;
 
    \draw [thick=1.5] (-1.4,-1)to[in=270,out=170] (-2.3,0)  ;
    
    \end{scope}

    \begin{scope}
   \draw [thick=1.5, ->-=.7] (-1.6,0.93) to[in=90,out=190] (-2.3,0);
    \draw [thick=1.5] (-1.4,1) to[in=135,out=10] (-0.3,0.3);
    \draw [thick=1.5]  (-1.6,-0.93)to[in=270,out=170] (-2.3,0);
    \draw [thick=1.5] (-1.4,-1)to[in=225,out=350] (-0.05,-0.05)  ;

    \draw [thick=1.5] (1.5,1) to[in=45,out=170] (0.3,0.3);
  
    \draw [thick=1.5]  (1.5,-1)to[in=220,out=10] (1.9,-0.8);
    \draw [thick=1.5]  (2.05,-0.65)to[in=270,out=40] (2.3,0);
  
    \draw [thick=1.5] (1.5,1) to[in=140,out=350] (1.9,0.8);
    \draw [thick=1.5] (2.05,0.65) to[in=90,out=320] (2.3,0);

    \draw [thick=1.5] (1.5,-1)to[in=315,out=190] (0.3,-0.3)  ;

    \draw [thick=1.5] (-0.3,0.3) to (0.3,-0.3);
    \draw [thick=1.5] (0.05,0.05) to (0.3,0.3);

    \node at (2.6,0.6) {\tiny $c^*_{a_2}$}; 
    \node at (1.9,1.2) {\tiny $c_{a_2}$};

    \end{scope}    
  \draw [thick=1.5] (1.95,0.7) to[in=90, out=230] (1.55,0.2);        
     \draw [thick=1.5] (1.55,0.2) to[in=230, out=270] (2.1,0.5);        
           
    \draw[thick=1.5] (1.95,0.7) to[in=160, out=50] (4.6cm,1);             
    \draw[thick=1.5] (2.2,0.6) to[in=160, out=50] (4.4cm,0.92);

    \draw [thick=1.5,blue] (1.95,-0.7) to[in=90, out=120] (1.55,-0.2);        
    \draw [thick=1.5,blue] (1.55,-0.2) to[in=120, out=270] (1.75,-0.8);

    \draw [thick=1.5,blue] (-4.5,-3.9) to[in=240,out=0] (-1.6,-0.93); 
    \draw [thick=1.5, blue] (-4.5,-4.1) to[in=240,out=0] (-1.4,-1); 
 
    \draw [thick=1.5, blue]  (4.7,-5.1)to[in=300,out=180] (1.85,-1);
    \draw [thick=1.5, blue]  (4.7,-4.9)to[in=300,out=180] (1.95,-0.7);
 
    \draw [thick=1.5, blue]  (4.7,-5.1)to (5.5,-4.9);
    \draw [thick=1.5, blue]  (4.7,-4.9)to (5.5,-5.1);

    \node at (-0.3,-0.35) {\tiny $\bigstar$};
    \node at (-0.25, -0.7) {\tiny $t_v$};

    \node at (0,0.3) {\tiny $\tau_v$};
    
    \node at (-5,-4) {\tiny $h_{a_3}$}; 
    \node at  (6,-5) {\tiny $h_{a_4}$}; 
    
    \node at (2.5,-0.8) {\tiny $c_{a_4}$}; 
    \node at (1.7,-1.3) {\tiny $c^*_{a_4}$}; 

    \node at (5.3, -5.3) {\tiny $\zeta_{a_4}$}; 

    \node at (0, -5.5) {\small $a_1 : v \to w_1$, \ \ $a_2: w_2 \to v$}; 
    \node at (0, -6) {\small $a_3 : v \to w_3$, \ \ $a_4: w_4 \to v$}; 
    \node at (0, -6.5) {\small $\partial \tau_v = -t_v z_{a_3} (e_v + c_{a_1}^* c_{a_1}) + (e_v + c_{a_2}
    c_{a_2}^*) (e_v + c_{a_4} c_{a_4}^*)$ }; 
 \node at (0, -7.1) {\small $\partial \zeta_{a_4} = (e_{w_4} + c_{a_4}^* c_{a_4}) - z_{a_4}$  }; 

		\end{tikzpicture}
		\caption{A local picture showing the choice of attachments and the corresponding computation of the differential.}
		\label{fig6}
	\end{figure}
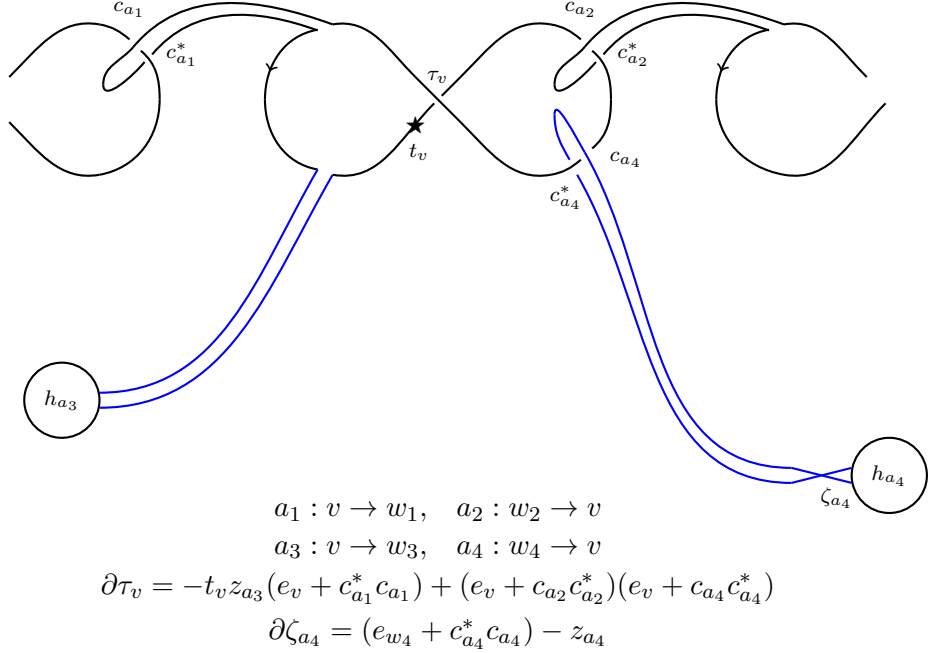
	
\begin{itemize}
\item a pair of generators $c_a$, $c_a^*$ with $|c_a|=|c_a^*|=0$ for each arrow $a$ of $\Gamma$,
	\item a generator $\tau_v$ with $|\tau_v|=-1$ and \footnote{Although our sign conventions for $CE^*$ follow \cite{EkNg}, as we must choose the null-cobordant spin structures on the components of $L_\Gamma$, we apply the substitution $t_v \mapsto -t_v$ to their formulas (see \cite[Rem2.7]{EkNg})}
	$$ \d \tau_v= \prod_{t(a)=v} (e_v+c_ac_a^*) -t_v \prod_{s(a)=v, \ a\notin T} z_a \prod_{s(a)=v, \ a\in T}  (e_v + c_a^*c_a) $$
	for each vertex $v$ of $\Gamma$, where $s(a)$ and $t(a)$ denote the source and target vertices of $a$, respectively,
	\item a generator $\zeta_a$ with $|\zeta_a|=-1$ and $$\d \zeta_a = e_v+c_a^*c_a-z_a $$ for each
        arrow $a$ in $\Gamma \backslash T$, where $v=s(a)$, 
	\item \emph{secondary generators} associated to crossings between nonplanar bands in the construction of $L_\Gamma$ 
\end{itemize}
We call the generators $t_v^\pm$, $z_a^\pm$ and those in the first three classes above \emph{primary
generators}. In the next lemma, we show that the secondary generators are really an artefact of the projection and could be removed. 

\begin{lem}\label{secondaries}
	The DG-subalgebra $\mathscr{P}_\Gamma$  generated by primary generators of $\mathscr{B}_\Gamma$ is quasi-isomorphic to $\mathscr{B}_\Gamma$.
\end{lem}
\begin{proof}
We will prove that $\mathscr{B}_\Gamma$ and $\mathscr{P}_\Gamma$ by showing that the secondary generators can be destabilised (after suitable elementary automorphisms) without effecting the rest of $\mathscr{B}_\Gamma$. 
Technically, such a stable tame isomorphism argument requires the DG-algebra to be semi-free, but this does not become an issue in the current set-up since we will only use elementary automorphisms which are identity on the non-free generators $t_v^\pm$ and $z_a^\pm$.

First of all, note that every intersection point of the arcs in the lower half of the standard form
    of $\Gamma$ (blue arcs in Fig.~(\ref{fig4})) induces a set of $4$ generators which correspond to the crossings between the associated bands in the resolution of $L_\Gamma$. 
We always assume that the band with the smaller slope goes under the other one.
A consequence of this is that a polygon whose all vertices are induced by intersection of arcs (in the standard form of $\Gamma$) has at least $2$ (Reeb-)positive vertices, hence do not contribute to the differential in $\mathscr{B}_\Gamma$. 
This is because at a vertex, the sign (of the quadrant in the interior of the polygon) is negative if and only if the slope decreases while traversing the polygon in the clockwise direction. 
Moreover, no secondary generator appears in the differential of a secondary generator induced by another intersection or a primary generator. 

In the rest of the proof we explain how to destabilise the secondary generators explicitly.
The generators induced by the same intersection point in the standard form of $\Gamma$ are destabilised simultaneously. 
This process varies slightly depending on whether the intersecting arcs end on the left or on the right of the diagram.
Examples of different cases are indicated in Fig.~(\ref{fig5}).
In the arguments below, the arrows in play are $a, a' \in \Gamma \setminus T$ from $i,i'$ to $j,j'$, resp., and the generators are denoted by $p,q,r,s$.

{\emph {Case 1: Both arcs end on the left:}} 
The differentials of the generators are
	$$ \d p = z_aq + rz_{a'} ,\ \ \d q=sz_{a'},  \ \ \d r= -z_as , \ \ \d s=0$$ 
Once we apply the following elementary automorphisms successively
$$ s \mapsto z_a^{-1} s , \ \ q \mapsto q-z_a^{-1} rz_{a'} , \ \  q \mapsto z^{-1}_aq$$ 	
the differentials become
$$ 	\d p = q , \ \ \d q=0, \ \ \d r= -s , \ \ \d s=0$$
and $p,q,r,s$ can be destabilised.

{\emph {Case 2: Both arcs on the right:}}
In this case, the 
differentials are
	$$ \d p = (e_i+c_a^*c_a)q + r(e_{i'}+c_{a'}^*c_{a'}) \ , \ \ \d q=s(e_{i'}+c_{a'}^*c_{a'}) \ ,  \ \ \d r= -(e_i+c_a^*c_a)s \ , \ \ \d s=0$$ 
	We apply the elementary automorphisms
	$$ p \mapsto p +\zeta_aq+r\zeta_{a'}+\zeta_a s\zeta_{a'}, \ \ q \mapsto q+ s\zeta_{a'} \ \ r \mapsto r-\zeta_as $$
	successively, making the differentials identical to those considered in the first case.
	
	{\emph {Case 3: The arcs end on different sides:}}
	The differentials are
	$$ \d p = z_{a}q+ r(e_{i'}+c_{a'}^*c_{a'}) \ , \ \ \d q= s(e_{i'}+c_{a'}^*c_{a'}) \ , \ \ \d r=z_{a}s \ , \ \ \d s=0$$
	In order to reduce to the first case, it suffices to apply the following elementary automorphisms successively
	$$ p \mapsto p +r\zeta_{a'} , \ \ q \mapsto q+s\zeta_{a'} $$
\end{proof}

\begin{thm}\label{derived}
	Let $\Gamma$ be a finite graph. The Chekanov-Eliashberg algebra $\mathscr{B}_\Gamma$ of the
Legendrian link $L_\Gamma$ and the derived multiplicative preprojective algebra $\mathscr{L}_\Gamma$
associated to $\Gamma$ are quasi-isomorphic DG-algebras. 
\end{thm}
\begin{proof}
By the previous lemma, it suffices to find a quasi-isomorphism between $\mathscr{P}_\Gamma$ and $\mathscr{L}_\Gamma$.
In fact, there is a DG-algebra isomorphism $\varphi : \mathscr{P}_\Gamma \to \mathscr{L}_\Gamma$ defined on the primary generators as identity except for 
$$\tau_v \mapsto \tau_v  + t_v  \left( \sum_{i=1}^{\iota_v} z_{a_1}  \cdots z_{a_{i-1}}\zeta_{a_i} (e_v + c_{a_{i+1}}c_{a^*_{i+1}})\cdots (e_v + c_{a_{\iota_v}}c_{a^*_{\iota_v}})\right) \prod_{s(a)=v, \ a\in T} (e_v + c_a^*c_a)$$
where $\{ a_1, \ldots , a_{\iota_v}\} $ is the (ordered) set of arrows with $s(a_i)=v$ and $a_i\in \Gamma \setminus T$.
\end{proof}

\section{Plumbings of cotangent bundles of surfaces of positive genus}
\label{g>0} 

In this section, we generalise the construction and the computation in the previous section to plumbings of cotangent bundles of surfaces of arbitrary genus.

\subsection{A Legendrian surgery presentation}

In Fig.~(\ref{fig7}), we give an example illustrating how to draw a Legendrian surgery picture of a
plumbing of cotangent bundles of surfaces of varying genus according to an arbitrary plumbing graph. 
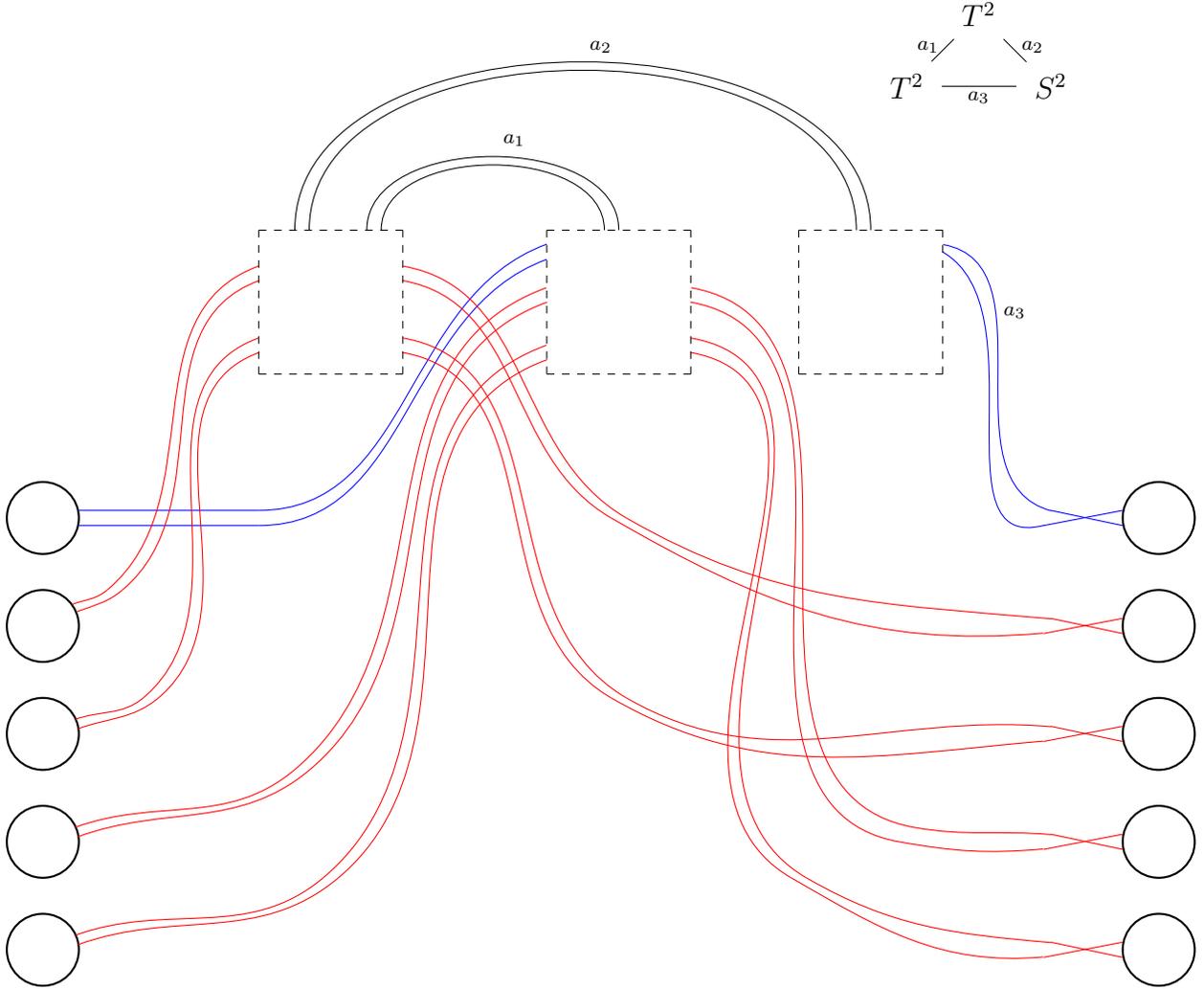
\begin{figure}[h!]
    \centering
    \begin{tikzpicture}[scale=1,auto=left,every node/.style={circle}]

    \tikzset{->-/.style={decoration={ markings,
                mark=at position #1 with {\arrow{>}}},postaction={decorate}}}

 \draw [thick] (-2,-2) circle (0.5); 
    \draw [thick] (-2,-3.5) circle (0.5); 
    \draw [thick] (-2,-5) circle (0.5); 
    \draw [thick] (-2,-6.5) circle (0.5); 
    \draw [thick] (-2,-8) circle (0.5); 

    \draw [thick] (13.5,-2) circle (0.5); 
    \draw [thick] (13.5,-3.5) circle (0.5); 
    \draw [thick] (13.5,-5) circle (0.5); 
    \draw [thick] (13.5,-6.5) circle (0.5); 
    \draw [thick] (13.5,-8) circle (0.5);

    \draw [blue] [yshift=-3](1,-2) to[in=200,out=0] (5,1.7); 
    \draw [blue] [yshift=-3](-1.5,-2) to[in=180,out=0] (1,-2); 
    
    \draw [blue] [yshift=3](1,-2) to[in=200,out=0] (5,1.7); 
    \draw [blue] [yshift=3](-1.5,-2) to[in=180,out=0] (1,-2); 
    
    \draw [blue][xshift=-3](12.02,-2.1) to[in=330,out=195] (10.6,1.7); 
    \draw [blue][xshift=3](11.9,-1.9) to[in=350,out=165] (10.4,1.8); 
   
    \draw [blue]  [yshift=3](12,-2) to [yshift=-6](13,-2);
    \draw [blue]  [yshift=-3](11.9,-2) to [yshift=6](13,-2);

    \draw [red] [xshift=-3](-1,-3)[in=200,out=40] to (1.1,1.5); 
    \draw [red] [xshift=3](-1,-3)[in=200,out=40] to (0.9,1.3);  

    \draw [red] [xshift=-3](-1,-3) to[out=220,in=20] (-1.49, -3.2); 
    \draw [red] [xshift=3](-1,-3) to[out=220,in=20] (-1.63, -3.3);

    \draw [red] [xshift=-3](6,-2)[in=350,out=150] to (3.1,1.3); 
    \draw [red] [xshift=3](6,-2)[in=350,out=150] to (2.9,1.5);  

    \draw [red] [xshift=-3](6,-2) [out=330,in=185] to (12,-3.6); 
    \draw [red] [xshift=3](6,-2) [out=330,in=175]  to (11.9,-3.4); 

    \draw [red]  [yshift=3](12,-3.5) to [yshift=-6](13,-3.5);
    \draw [red]  [yshift=-3](11.9,-3.5) to [yshift=6](13,-3.5);

    \draw [red] [xshift=-3](-0.5,-4.5)[in=200,out=40] to (1.1,0.5); 
    \draw [red] [xshift=3](-0.5,-4.5)[in=200,out=40] to (0.9,0.3);  

    \draw [red] [xshift=-3](-0.5,-4.5) to[out=220,in=20] (-1.45, -4.8); 
    \draw [red] [xshift=3](-0.5,-4.5) to[out=220,in=20] (-1.62, -4.93);

    \draw [red] [xshift=-3](6,-4.5)[in=350,out=150] to (3.1,0.3); 
    \draw [red] [xshift=3](6,-4.5)[in=350,out=150] to (2.9,0.5);  

    \draw [red] [xshift=-3](6,-4.5) [out=330,in=185] to (12,-5.1); 
    \draw [red] [xshift=3](6,-4.5) [out=330,in=175]  to (11.9,-4.9);

    \draw [red]  [yshift=3](12,-5) to [yshift=-6](13,-5);
    \draw [red]  [yshift=-3](11.9,-5) to [yshift=6](13,-5);

     \draw [red] [xshift=-3](1.6,-5.5)[in=200,out=40] to (5.1,1.2); 
    \draw [red] [xshift=3](1.6,-5.5)[in=200,out=40] to (4.9,1);  

    \draw [red] [xshift=-3](1.6,-5.5) to[out=220,in=20] (-1.45, -6.3); 
    \draw [red] [xshift=3](1.6,-5.5) to[out=220,in=20] (-1.62, -6.43);

    \draw [red] [xshift=-3](10,-6.5)[in=350,out=170] to (7.1,1); 
    \draw [red] [xshift=3](10,-6.3)[in=350,out=170] to (6.9,1.2);  

    \draw [red] [xshift=-3](10,-6.5) [out=350,in=185] to (12,-6.6); 
    \draw [red] [xshift=3](10,-6.3) [out=350,in=175]  to (11.9,-6.4); 

    \draw [red]  [yshift=3](12,-6.5) to [yshift=-6](13,-6.5);
    \draw [red]  [yshift=-3](11.9,-6.5) to [yshift=6](13,-6.5);

     \draw [red] [xshift=-3](2,-7)[in=200,out=40] to (5.1,0.4); 
    \draw [red] [xshift=3](2,-7)[in=200,out=40] to (4.9,0.2);  

    \draw [red] [xshift=-3](2,-7) to[out=220,in=20] (-1.45, -7.8); 
    \draw [red] [xshift=3](2,-7) to[out=220,in=20] (-1.62, -7.93);

    \draw [red] [xshift=-3](8.5,-7)[in=350,out=150] to (7.1,0.3); 
    \draw [red] [xshift=3](8.5,-7)[in=350,out=150] to (6.9,0.5);  

    \draw [red] [xshift=-3](8.5,-7) [out=330,in=185] to (12,-8.1); 
    \draw [red] [xshift=3](8.5,-7) [out=330,in=175]  to (11.9,-7.9); 

    \draw [red]  [yshift=3](12,-8) to [yshift=-6](13,-8);
    \draw [red]  [yshift=-3](11.9,-8) to [yshift=6](13,-8);

    \draw (1.5,2)[in=90,out=90] to (9.5,2);
    \draw (1.7,2)[in=90,out=90] to (9.3,2);
 
    \draw (2.5,2)[in=90,out=90] to (6,2);
    \draw (2.7,2)[in=90,out=90] to (5.8,2);

\draw[dashed]  (1,0) -- (3,0);
\draw[dashed]  (5,0) -- (7,0);
\draw[dashed]  (8.5,0) -- (10.5,0);

\draw[dashed]  (1,2) -- (3,2);
\draw[dashed]  (5,2) -- (7,2);
\draw[dashed]  (8.5,2) -- (10.5,2);

\draw[dashed]  (1,0) -- (1,2);
\draw[dashed]  (5,0) -- (5,2);
\draw[dashed]  (8.5,0) -- (8.5,2);

\draw[dashed] (3,0) -- (3,2);
\draw [dashed] (7,0) -- (7,2);
\draw [dashed] (10.5,0) -- (10.5,2);

   \node (n1) at (10,4) {$T^2$};
       \node (n2) at (11,5)  {$T^2$};
  \node (n3) at (12,4)  {$S^2$};

  \foreach \to/\from in {n2/n1,n3/n1}
    \draw (\from) -- (\to);
    
  \draw (n2) -- (n3); 
  
\node at (10.3,4.55)  {\tiny{$a_1$}};
\node at (11.75,4.55)  {\tiny{$a_2$}};
\node at (11,3.85)  {\tiny{$a_3$}};
  
\node at (4.55,3.25)  {\tiny{$a_1$}};
\node at (5.75,4.55)  {\tiny{$a_2$}};
\node at (11.5,0.85)  {\tiny{$a_3$}};

    \end{tikzpicture}
    \caption{Surgery diagram for plumbings of cotangent bundles of surfaces of varying genus.
    The drawn example shows a plumbing of two copies of $T^*T^2$ with one copy of $T^*S^2$. } 
    \label{fig7}

\end{figure}
Inside the boxes, we put a standard surgery picture of the cotangent bundle of a genus $g_v$
surface. This is described in Fig.~(\ref{thebeast}). It is obtained by band summing the 2-handles of $g_v$
copies of the standard picture $T^*T^2$ given in Fig.~(\ref{gompf}) and then applying simplifying
Reidemeister moves.  There is a freedom in choosing the location of the band sum. Our choice makes
it easier to see directly that the resulting DG-algebra is quasi-isomorphic to the derived version of
the higher genus multiplicative preprojective algebra.

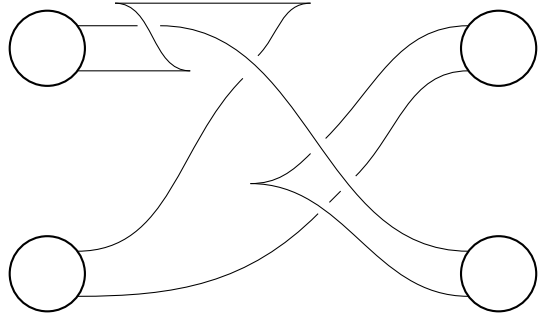
\begin{figure}[h!]
    \centering
    \begin{tikzpicture}[scale=1,auto=left,every node/.style={circle}]
   \tikzset{->-/.style={decoration={ markings, mark=at position #1 with {\arrow{>}}},postaction={decorate}}}

     \draw [thick] (-2,0) circle (0.5); 
    \draw [thick] (-2,3) circle (0.5); 
    \draw [thick] (4,0) circle (0.5); 
    \draw [thick] (4,3) circle (0.5);  

        \draw (-0.5,3.3) to[in=180,out=0] (3.6,0.3); 
        \draw (-1.6,3.3) to (-0.8,3.3); 
        \draw (-1.6,2.7) to (-0.1,2.7); 
        \draw (-0.1,2.7) to[out=180,in=0] (-1.1,3.6); 
       \draw (-1.1,3.6) to (1.5,3.6); 
       \draw (1.5,3.6) to [out=180,in=45] (0.8,2.9); 
       \draw (0.6,2.6) to [out=225,in=0] (-1.6,0.3); 

       \draw (3.6,-0.3) to [in=0, out=180] (0.7,1.2);
       \draw (0.7, 1.2) to [in=225, out=0] (1.5,1.6);
       \draw (1.7,1.8) to [in=180, out=45] (3.6,3.3);

       \draw (-1.6,-0.3) to [in=225, out=0] (1.6,0.8);
        \draw (1.75,0.95) to[in=225,out=45] (1.9,1.1); 
       \draw (2.1,1.3) to[in=180,out=45] (3.6,2.7); 

    \end{tikzpicture}
    \caption{Legendrian surgery picture of $T^*T^2$}
    \label{gompf} 
\end{figure}

\subsection{Computation of the Chekanov-Eliashberg DG-algebra}

As shown in Lem.~(\ref{secondaries}) the secondary generators induced by intersecting blue curves
as in
Fig.~(\ref{fig5}) can be destabilised without effecting the rest of the DG-algebra
$\mathscr{B}_{\Gamma, {\mathbf g}}$. In the positive genus case, we have a new set of secondary
generators induced by intersecting red and blue curves as in Fig.~(\ref{fig7}) which can be destabilised exactly as before.
Another set of secondary generators  $p_{v,k},q_{v,k},r_{v,k}$ and $s_{v,k}$, for $k=1, \dots,
g_v-1$, indicated by quadruples $(p_i,q_i,r_i,s_i)$ for $i=1,2, \ldots,g-1$ in Fig.~(\ref{thebeast}),
have a configuration and destabilisation similar to those treated in the first case of the proof of Lem.~(\ref{secondaries}).
By these initial simplifications, $\mathscr{B}_{\Gamma, {\mathbf g}}$ is quasi-isomorphic to its DG-subalgebra $\mathscr{P}_{\Gamma, {\mathbf g}}$ generated by the primary generators below:
\begin{itemize}
	\item generators $t_v^\pm$ for each vertex $v$,
	\item generators $\zeta_a$, $z_a^\pm$ for each arrow $a$ in $\Gamma \backslash T$, 
	\item generators $c_a$, $c_a^*$ for each arrow $a$ in $\Gamma$,
	\item generators $\tau_v$, $z^\pm_{v,i}$, $\xi_{v,j}$, $x_{v,j}$ for each vertex $v$, $i=1, \dots , 2g_v$, and $j=1,\dots , 4g_v$
\end{itemize}
For the last group of generators above, see Fig.~(\ref{thebeast}) where the notation is simplified by using $x_i$ instead of $x_{v,i}$, etc. 
Note that $z_{v,i}^\pm$ correspond to the generators of the internal DG-algebra $\mathscr{I}_2 ^\K \simeq \mathbb{K}[z^\pm]$ of the $i^{th}$ $1$-handle $h_i$ in the diagram of the cotangent bundle of the surface of genus $g_v$ in Fig.~(\ref{thebeast}).

As in the genus $0$ case, the DG-algebra $\mathscr{P}_{\Gamma, {\mathbf g}}$ of primary generators is supported in nonpositive gradings with $\tau_v$, $\zeta_a$ and $\xi_{v,j}$ in grading $-1$ and the rest of the generators in grading $0$. 
Note that although the generators $\xi_{v,4i-2}, \xi_{v,4i-1}$, for $i = 1, \dots , g_v$, are induced by the twists near the right feet of the corresponding $1$-handles similar to the generators of type $\zeta_a$, they can be destabilised and will not appear in the ultimate presentation of $\mathscr{B}_{\Gamma, {\mathbf g}}$.

Before we show how to destabilise the generators $\xi_{v,j}$ and $x_{v,j}$, let us describe the differential on $\mathscr{P}_{\Gamma, {\mathbf g}}$ as given by Fig.~(\ref{thebeast}) where $z_i=z_{v,i}, \xi_j=\xi_{v,j}, x_j=x_{v,j}$, $g=g_v$:
\begin{align*}
 \d \tau_v&= \prod_{t(a)=v} (e_v+c_ac_a^*) -t_v x_1\prod_{s(a)=v, \ a\notin T} z_a \prod_{s(a)=v, \ a\in T}  (e_v + c_a^*c_a) \\
\d \zeta_a &= e_{s(a)}+c_a^*c_a-z_a \\
\d \xi_{4k-3} &= -z_{2k}x_{4k-3}+x_{4k-2}x_{4k-1}x_{4k}, \ \text{for} \ k=1,\dots, g-1 \\ 
\d \xi_{4k-2}&=x_{4k-2}-z_{2k-1}, \ \text{for} \ k=1,\dots, g \\ 
\d \xi_{4k-1}&=x_{4k-1}-z_{2k}, \ \text{for} \ k=1,\dots, g \\ 
\d \xi_{4k}&=e_v-x_{4k+1}z_{2k-1}x_{4k}, \ \text{for} \ k=1,\dots, g-2 \\ 
\d \xi_{4g-3} &= -z_{2g}x_{4g-3}+x_{4g-2}x_{4g-1} \\ 
\d \xi_{4g-4}&=e_v-x_{4g-3}x_{4g}z_{2g-3}x_{4g-4}\\
\d \xi_{4g}&=e_v-z_{2g-1}x_{4g}
\end{align*}

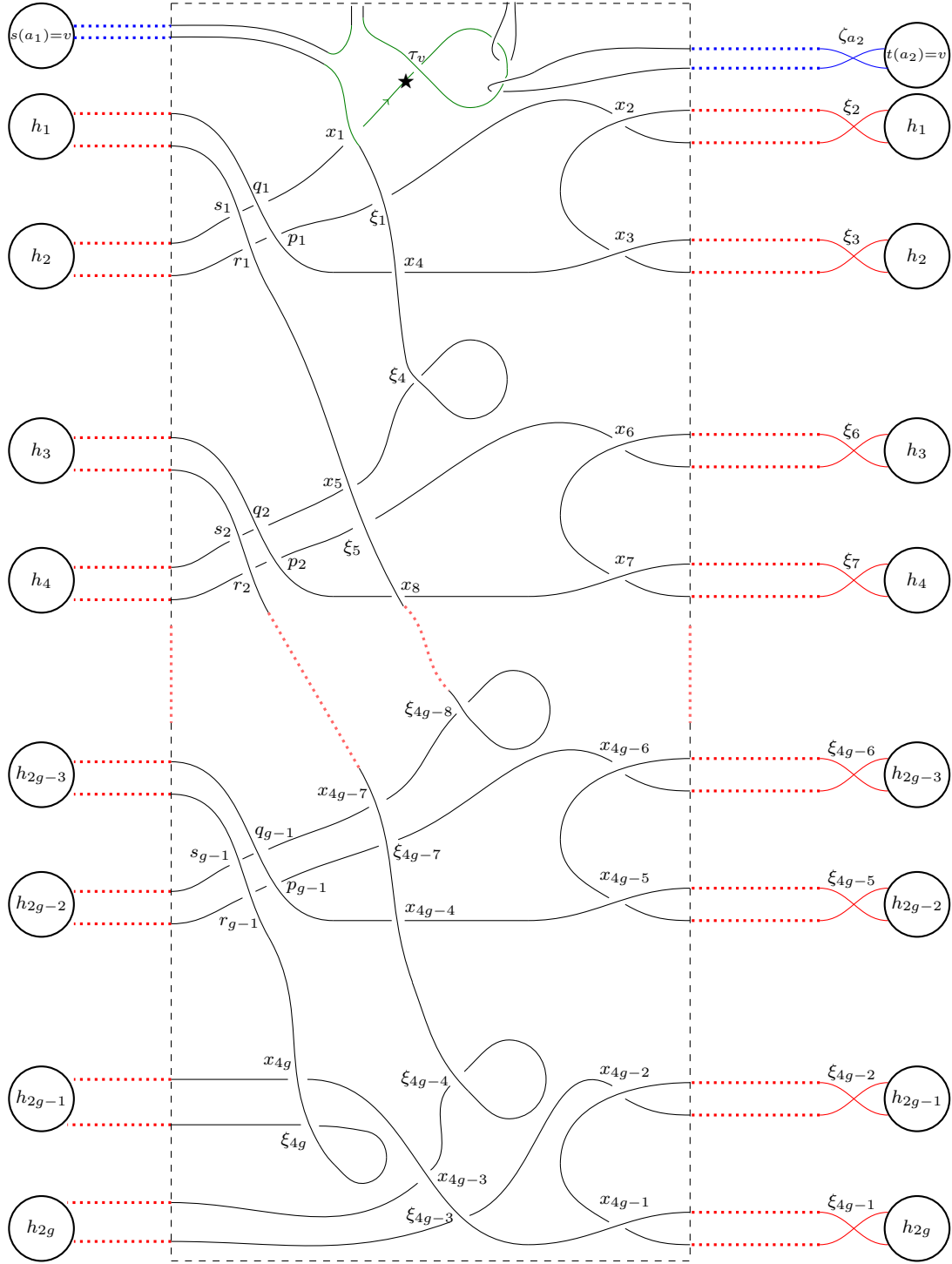
\begin{figure}[htb!]
    \centering
    \begin{tikzpicture}[scale=0.99,auto=left,every node/.style={circle}]
   \tikzset{->-/.style={decoration={ markings,
                mark=at position #1 with {\arrow{>}}},postaction={decorate}}}

    \begin{scope}[scale=0.6, xshift=8cm, yshift=4cm]           
    
   \draw   (-2.3,0.4)to[in=0, out=150] (-6.4,1.1);
   \draw  (-2.4,0.1)to[in=0, out=150] (-6.4,0.8);

    \draw (-1.7,1.1) to (-1.7,1.6); 
    \draw (-1.4,1.3) to (-1.4,1.6); 
    
    \draw[green!50!black]  (-1.7,1.1) to[in=330,out=270] (-2.3,0.4);
   
    \draw[green!50!black]  (-1.4,1.3) to[in=135,out=270] (-0.3,0.3);
   \draw[green!50!black]  (-1.5,-2)to[in=330,out=130] (-2.4,0.1);
    \draw[->-=.5, green!50!black] (-1.4,-1.5)to[in=225,out=45] (-0.1,-0.1);
 
    \draw[green!50!black]  (1.5,-1)to[in=245,out=10] (2.2,-0.4);
    \draw[green!50!black]   (2.25,-0.25)to[in=270,out=40] (2.3,0);
 
    \draw[green!50!black]  (1.5,1) to[in=45,out=170] (0.3,0.3);
    \draw[green!50!black]  (1.5,1) to[in=140,out=350] (1.9,0.8);
   \draw[green!50!black]  (2.05,0.65) to[in=90,out=320] (2.3,0);

    \draw[green!50!black]  (1.5,-1)to[in=315,out=190] (0.3,-0.3)  ;

    \draw[green!50!black]  (-0.3,0.3) to (0.3,-0.3);
    \draw[green!50!black]  (0.1,0.1) to (0.3,0.3);

    \draw (2.3,1.7) to[in=150,out=270] (2.1, 0.2); 
    \draw (2.5,1.7) to[in=70,out=270] (2.4, 0.2); 

    \draw (3,-0.1) to[in=150,out=210] (1.9, -0.6); 
    \draw (3,-0.1) to[in=180,out=30] (7,0.5); 
    
        \draw (2.2, -0.6) to [in=180,out=0] (7,0);

    \node at (0,0.3) {\tiny{$\tau_v$}}; 
    \node at (-0.3,-0.35) {\tiny $\bigstar$};

        \node at (-2.1, -1.7) {\tiny{$x_1$}};

    \end{scope}
            
    \begin{scope}[scale=0.6, xshift=8cm, yshift=-4cm]

    \draw  (1.5,-1)to[in=270,out=10] (2.3,0);
    \draw  (1.5,1) to[in=45,out=170] (0.3,0.3);
    \draw  (1.5,1) to[in=90,out=350] (2.3,0);

    \draw  (1.5,-1)to[in=315,out=190] (0.3,-0.3)  ;

    \draw  (-0.3,0.5) to[in=135,out=280] (0.3,-0.3);
    \draw  (0.1,0.1) to (0.3,0.3);

    \draw (-1.55,-2.7)to[in=225,out=30] (-0.1,-0.1);

    \node at (-0.5,0.1) {\tiny{$\xi_4$}};

    \end{scope}

    \begin{scope}[scale=0.6, xshift=9.1cm, yshift=-12.5cm]

    \draw  (1.5,-1)to[in=270,out=10] (2.3,0);
    \draw  (1.5,1) to[in=45,out=170] (0.3,0.3);
    \draw  (1.5,1) to[in=90,out=350] (2.3,0);

    \draw  (1.5,-1)to[in=315,out=190] (0.3,-0.3)  ;

    \draw  (-0.3,0.5) to[in=135,out=315] (0.3,-0.3);
    \draw  (0.2,0.2) to (0.3,0.3);
   
    \draw (-1.9,-2.3)to[in=225,out=30] (-0.1,-0.1);

    \node at (-0.8,-0.0) {\tiny{$\xi_{4g-8}$}};

    \end{scope}

    \begin{scope}[scale=0.6, xshift=9cm, yshift=-22cm]

    \draw  (1.5,-1)to[in=270,out=10] (2.3,0);
    \draw  (1.5,1) to[in=45,out=170] (0.3,0.3);
    \draw  (1.5,1) to[in=90,out=350] (2.3,0);

    \draw  (1.5,-1)to[in=315,out=190] (0.3,-0.3)  ;

    \draw  (0.2,0.2) to (0.3,0.3);

    \draw (0.3,-0.3) to [in=290,out=135] (-1.14,2.5); 

    \node at (-0.8,-0.0) {\tiny{$\xi_{4g-4}$}};
 
    \end{scope}

   \begin{scope}[scale=0.4, xshift=8.5cm, yshift=-36cm]

   \draw  (1.5,-1)to[in=270,out=10] (2.3,0);
   \draw  (1.5,1) to[in=0,out=170] (-0.3,1.2);
   \draw  (1.5,1) to[in=90,out=350] (2.3,0);

   \draw  (1.5,-1)to[in=315,out=190] (0.3,-0.3)  ;
   \draw  (-0.3,0.5) to[in=135,out=300] (0.3,-0.3);

   \end{scope}

\draw (3.9,1.2) to[in=100,out=300] (4.62,-2.1);

\node at (4.2, 0.1) {\tiny{$\xi_1$}}; 
\node at (4.75, -0.6) {\tiny{$x_4$}}; 

\node at (3.5, -4.0) {\tiny{$x_5$}}; 
\node at (3.8, -5.0) {\tiny{$\xi_5$}}; 

\node at (4.7, -5.6) {\tiny{$x_8$}};

\node at (2.95, -0.25)  {\tiny{$p_1$}}; 
\node at (2.4, 0.55) {\tiny{$q_1$}}; 
\node at (2.1, -0.6) {\tiny{$r_1$}}; 
\node at (1.8, 0.25) {\tiny{$s_1$}}; 

\node at (2.95, -5.25)  {\tiny{$p_2$}}; 
\node at (2.4, -4.45) {\tiny{$q_2$}}; 
\node at (2.1, -5.6) {\tiny{$r_2$}}; 
\node at (1.8, -4.75) {\tiny{$s_2$}};

\node at (3.1, -10.25)  {\tiny{$p_{g-1}$}}; 
\node at (2.6, -9.4) {\tiny{$q_{g-1}$}}; 
\node at (2.05, -10.8) {\tiny{$r_{g-1}$}}; 
\node at (1.6, -9.75) {\tiny{$s_{g-1}$}};

\draw[yshift=-9.6cm] (3.9,1.2) to[in=110,out=300] (4.72,-2.1);

\draw(1, -0.3) to [in=205, out=0] (1.9, 0.1); 
\draw(2.1, 0.2) to (2.27, 0.28); 
\draw(1, -0.8) to [in=205,out=0] (2.1, -0.4); 
\draw(2.3, -0.3) to (2.47, -0.22);

\draw (3.5, -0.75) to [in=0, out=180] (1,1.7); 
\draw (2.5, -1) to [in=0, out=120] (1,1.2); 
\draw (3.5,-0.75) to [in=180, out=0] (4.4,-0.75);
\draw (2.5,-1) to [in=120, out=300] (4.58,-5.9); 

\draw (4.1, 0.3) to [in=25,out=210] (2.7, -0.15);
\draw (3.65, 1.2) to [in=25,out=225] (2.5, 0.35);

\draw (6.5, -0.75) to (4.6,-0.75);

\draw (7.8, 1.7) to [in=30,out=150] (4.4, 0.45);

\draw [yshift=-5cm] (1, -0.3) to [in=205, out=0] (1.9, 0.1); 
\draw [yshift=-5cm] (2.1, 0.2) to (2.27, 0.28); 
\draw [yshift=-5cm] (1, -0.8) to [in=205, out=0] (2.1, -0.4); 
\draw [yshift=-5cm] (2.3, -0.3) to (2.47, -0.22);

\draw  [yshift=-5cm] (3.5, -0.75) to [in=0, out=180] (1,1.7); 
\draw  [yshift=-5cm] (2.5, -1) to [in=0, out=120] (1,1.2); 
\draw [yshift=-5cm] (3.5,-0.75) to [in=180, out=0] (4.4,-0.75);
\draw [yshift=-5cm] (3.8, 0.3) to [in=25,out=210] (2.7, -0.15);
\draw [yshift=-5cm] (3.65, 0.9) to [in=25,out=210] (2.5, 0.35);

\draw [yshift=-5cm] (6.5, -0.75) to (4.6,-0.75);

\draw [yshift=-5cm] (7.8, 1.7) to [in=30,out=150] (4.15, 0.45);

\draw [yshift=-10cm] (1, -0.3) to [in=205, out=0] (1.9, 0.1); 
\draw [yshift=-10cm] (2.1, 0.2) to (2.27, 0.28); 
\draw [yshift=-10cm] (1, -0.8) to [in=205, out=0] (2.1, -0.4); 
\draw [yshift=-10cm] (2.3, -0.3) to (2.47, -0.22);

\draw  [yshift=-10cm] (3.5, -0.75) to [in=0, out=180] (1,1.7); 
\draw [yshift=-10cm]  (2.5, -1) to [in=0, out=120] (1,1.2); 
\draw [yshift=-10cm] (3.5,-0.75) to [in=180, out=0] (4.4,-0.75);
\draw [yshift=-10cm] (2.5,-1) to [in=120, out=300] (3.28,-4.2); 

\draw [yshift=-10cm] (4.2, 0.4) to [in=25,out=200] (2.7, -0.15);
\draw [yshift=-10cm] (4.05, 1) to [in=25,out=210] (2.5, 0.35);

\draw [yshift=-10cm] (6.5, -0.75) to (4.6,-0.75);

\draw [yshift=-10cm](7.8, 1.7) to [in=20,out=150] (4.5, 0.5);

\draw [very thick, dotted, red!60, yshift=-5cm] (2.5,-1) to [in=120, out=300] (3.9,-3.4); 
\draw [very thick, dotted, red!60, xshift=0.1cm, yshift=-4cm] (4.5,-1.9) to[in=135,out=315] (5.18,-3.2);

\begin{scope}[yshift=-0.5cm]

\draw [thick] (-1,3.4) circle (0.5); 
\draw [thick] (12.5,3.1) circle (0.5); 

    \node at (-1.05,3.4) {\tiny{${\ }_{s(a_1)=v}$}};
    \node at (12.45,3.1)  {\tiny{${\ }_{t(a_2)=v}$}};


\draw [thick] (-1,0) circle (0.5); 
\draw [thick] (-1,2) circle (0.5); 
\node at (-1,0) {\tiny{$h_2$}};
\node at (-1,2) {\tiny{$h_1$}};
\draw [thick] (12.5,0) circle (0.5); 
\draw [thick] (12.5,2) circle (0.5);  
\node at (12.5,0) {\tiny{$h_2$}};
\node at (12.5,2) {\tiny{$h_1$}};

\draw (9,2.25) [in=90, out= 180] to (7,1);
\draw (9,1.75) [in=330, out=180] to (8, 2);
\draw (7,1) [in= 150, out=270] to (7.75,0.1);
\draw (9,-0.25) [in=330, out=180] to (8, 0);
\draw (9,0.25) [in=0, out= 180] to (6.5,-0.25);

\draw [very thick, blue, dotted] (-0.5,3.55) to (1,3.55);  
\draw [very thick, blue, dotted] (-0.5,3.37) to (1,3.37);  

\draw [very thick, blue, dotted] (11,3.2) to (9,3.2);  
\draw [very thick, blue, dotted] (11,2.9) to (9,2.9);  

\draw [blue] (11,2.9) to[in=180,out=0] (12, 3.2);
\draw [blue] (11,3.2) to[in=180,out=0] (12.05, 2.9);

\node at (11.5,3.4) {\tiny{$\zeta_{a_2}$}};

\draw [very thick, red, dotted] (1,2.2) to (-0.5,2.2);  
\draw [very thick, red, dotted] (1,1.7) to (-0.5,1.7);  

\draw [very thick, red, dotted] (11,2.25) to (9,2.25);  
\draw [very thick, red, dotted] (11,1.75) to (9,1.75);  

\draw [red] (11,2.25) to[in=180,out=0] (12.05, 1.75);
\draw [red] (11,1.75) to[in=180,out=0] (12.05, 2.25);

\draw [very thick, red, dotted] (1,0.2) to (-0.5,0.2);  
\draw [very thick, red, dotted] (1,-0.3) to (-0.5,-0.3);  

\draw [very thick, red, dotted] (11,0.25) to (9,0.25);  
\draw [very thick, red, dotted] (11,-0.25) to (9,-0.25);  

\draw [red] (11,0.25) to[in=180,out=0] (12.05, -0.25);
\draw [red] (11,-0.25) to[in=180,out=0] (12.05, 0.25);

\node at (11.5, 2.3) {\tiny{$\xi_2$}}; 
\node at (11.5, 0.3) {\tiny{$\xi_3$}}; 
\node at (8, 2.3) {\tiny{$x_2$}}; 
\node at (8, 0.3) {\tiny{$x_3$}}; 

\draw [thick] (-1,-3) circle (0.5); 
\draw [thick] (-1,-5) circle (0.5); 
\node at (-1,-3)  {\tiny{$h_3$}};
\node at  (-1,-5) {\tiny{$h_4$}};
\draw [thick] (12.5,-3) circle (0.5); 
\draw [thick] (12.5,-5) circle (0.5); 
\node at (12.5,-3) {\tiny{$h_3$}};
\node at  (12.5,-5) {\tiny{$h_4$}}; 

\draw (9,-2.75) [in=90, out= 180] to (7,-4);
\draw (9,-3.25) [in=330, out=180] to (8, -3);
\draw (7,-4) [in= 150, out=270] to (7.75,-4.9);
\draw (9,-5.25) [in=330, out=180] to (8, -5);
\draw (9,-4.75) [in=0, out= 180] to (6.5,-5.25);

\draw [very thick, red, dotted] (1,-2.8)  to (-0.5,-2.8);  
\draw [very thick, red, dotted] (1,-3.3)  to (-0.5,-3.3);  

\draw [very thick, red, dotted] (11,-2.75)  to (9,-2.75) ;  
\draw [very thick, red, dotted] (11,-3.25)  to (9,-3.25);  

\draw [red] (11,-2.75) to[in=180,out=0] (12.05, -3.25) ;
\draw [red] (11,-3.25) to[in=180,out=0] (12.05, -2.75) ;

\draw [very thick, red, dotted] (1,-4.8)  to (-0.5,-4.8);  
\draw [very thick, red, dotted] (1,-5.3)  to (-0.5,-5.3);  

\draw [very thick, red, dotted] (11,-4.75) to (9,-4.75);  
\draw [very thick, red, dotted] (11,-5.25) to (9,-5.25);  

\draw [red] (11,-4.75)  to[in=180,out=0] (12.05, -5.25);
\draw [red] (11,-5.25)  to[in=180,out=0] (12.05, -4.75);

\node at (11.5, -2.7) {\tiny{$\xi_6$}}; 
\node at (11.5, -4.7) {\tiny{$\xi_7$}}; 
\node at (8, -2.7) {\tiny{$x_6$}}; 
\node at (8, -4.7) {\tiny{$x_7$}};

 \draw [thick] (-1,-8) circle (0.5); 
    \draw [thick] (-1,-10) circle (0.5); 
    \node at (-1,-8)  {\tiny{$h_{2g-3}$}};
\node at (-1,-10)  {\tiny{$h_{2g-2}$}}; 
    \draw [thick] (12.5,-8) circle (0.5); 
    \draw [thick] (12.5,-10) circle (0.5);  
        \node at (12.5,-8)  {\tiny{$h_{2g-3}$}};
\node at (12.5,-10)   {\tiny{$h_{2g-2}$}}; 

\draw (9,-7.75) [in=90, out= 180] to (7,-9);
\draw (9,-8.25) [in=330, out=180] to (8, -8);
\draw (7,-9) [in= 150, out=270] to (7.75,-9.9);
\draw (9,-10.25) [in=330, out=180] to (8, -10);
\draw (9,-9.75) [in=0, out= 180] to (6.5,-10.25);

\draw [very thick, red, dotted] (1,-7.8)  to (-0.5,-7.8);  
\draw [very thick, red, dotted] (1,-8.3)  to (-0.5,-8.3);  

\draw [very thick, red, dotted] (11,-7.75)  to (9,-7.75) ;  
\draw [very thick, red, dotted] (11,-8.25)  to (9,-8.25);  

\draw [red] (11,-7.75) to[in=180,out=0] (12.05, -8.25) ;
\draw [red] (11,-8.25) to[in=180,out=0] (12.05, -7.75) ;

\draw [very thick, red, dotted] (1,-9.8)  to (-0.5,-9.8);  
\draw [very thick, red, dotted] (1,-10.3)  to (-0.5,-10.3);  

\draw [very thick, red, dotted] (11,-9.75) to (9,-9.75);  
\draw [very thick, red, dotted] (11,-10.25) to (9,-10.25);  

\draw [red] (11,-9.75)  to[in=180,out=0] (12.05, -10.25);
\draw [red] (11,-10.25)  to[in=180,out=0] (12.05, -9.75);

\node at (11.5, -7.6) {\tiny{$\xi_{4g-6}$}}; 
\node at (11.5, -9.6) {\tiny{$\xi_{4g-5}$}}; 
\node at (8, -7.6) {\tiny{$x_{4g-6}$}}; 
\node at (8, -9.6) {\tiny{$x_{4g-5}$}}; 

\node at (3.65, -8.3) {\tiny{$x_{4g-7}$}}; 
\node at (4.8, -9.2) {\tiny{$\xi_{4g-7}$}}; 

\node at (5, -10.1) {\tiny{$x_{4g-4}$}};

    \draw [thick] (-1,-13) circle (0.5); 
    \draw [thick] (-1,-15) circle (0.5); 
            \node at (-1,-13)  {\tiny{$h_{2g-1}$}};
\node at  (-1,-15)    {\tiny{$h_{2g}$}}; 
    \draw [thick] (12.5,-13) circle (0.5); 
    \draw [thick] (12.5,-15) circle (0.5);  
           \node at (12.5,-13) {\tiny{$h_{2g-1}$}};
\node at   (12.5,-15)    {\tiny{$h_{2g}$}}; 

     \draw (9,-12.75) [in=90, out= 180] to (7,-14);
\draw (9,-13.25) [in=330, out=180] to (8, -13);
\draw (7,-14) [in= 150, out=270] to (7.75,-14.9);
\draw (9,-15.25) [in=330, out=180] to (8, -15);
\draw (9,-14.75) [in=0, out= 180] to (6.5,-15.25);

\draw (6.5, -15.25) to [in=0, out=180] (3.1, -12.7);

\draw [very thick, red, dotted] (1,-12.7)  to (-0.5,-12.7);  
\draw [very thick, red, dotted] (1,-13.4)  to (-0.6,-13.4);  

\draw [very thick, red, dotted] (11,-12.75)  to (9,-12.75) ;  
\draw [very thick, red, dotted] (11,-13.25)  to (9,-13.25);  

\draw [red] (11,-12.75) to[in=180,out=0] (12.05, -13.25) ;
\draw [red] (11,-13.25) to[in=180,out=0] (12.05, -12.75) ;

\draw [very thick, red, dotted] (1,-14.6)  to (-0.6,-14.6);  
\draw [very thick, red, dotted] (1,-15.2)  to (-0.5,-15.2);  

\draw [very thick, red, dotted] (11,-14.75) to (9,-14.75);  
\draw [very thick, red, dotted] (11,-15.25) to (9,-15.25);  

\draw [red] (11,-14.75)  to[in=180,out=0] (12.05, -15.25);
\draw [red] (11,-15.25)  to[in=180,out=0] (12.05, -14.75);

\end{scope}

\draw [yshift=-15cm](7.8, 1.7) to [in=20,out=150] (5.6, -0.3);
\draw [yshift=-15cm](5.3, -0.4) to [in=0,out=200] (1, -0.7); 
\draw [yshift=-15cm](4.8, 0.2) to [in=0,out=220] (1, -0.1); 
\draw [yshift=-15cm](5, 0.4) to [in=225, out=40] (5.3,1.7);

\node at (5.5, 0.4) [yshift=-15cm] {\tiny{$x_{4g-3}$}};
\node at (5.0, -0.15) [yshift=-15cm] {\tiny{$\xi_{4g-3}$}};

\draw[yshift=-15cm](1,1.1) to (3,1.1);
\draw[yshift=-15cm](1,1.8) to (2.8,1.8);

\node at (2.9,1) [yshift=-15cm] {\tiny{$\xi_{4g}$}}; 
\node at (2.65,2.2) [yshift=-15cm] {\tiny{$x_{4g}$}}; 

\node at (11.5, -13.1) {\tiny{$\xi_{4g-2}$}}; 
\node at (11.5, -15.1) {\tiny{$\xi_{4g-1}$}}; 
\node at (8.0, -13.1) {\tiny{$x_{4g-2}$}}; 
\node at (8.0, -15.1) {\tiny{$x_{4g-1}$}};

\draw[dashed] (1,3.4) to (9,3.4);
\draw[dashed] (1,3.4) to (1,-6.2);
\draw[very thick, dotted, red!60] (1,-6.2) to (1,-7.7);
\draw[dashed] (1,-7.7) to (1,-16);

\draw[dashed] (9,3.4) to (9,-6.2);
\draw[very thick, dotted, red!60] (9,-6.2) to (9,-7.7);
\draw[dashed] (9,-7.7) to (9,-16);

\draw[dashed] (1,-16) to (9,-16);

    \end{tikzpicture}
\caption{The green part of the 2-handle of $T^*\Sigma_g$ is analogous to Fig.~(\ref{fig6})}
  \label{thebeast} 
\end{figure} 

The following lemma suffices to generalise the quasi-isomorphism between $\mathscr{B}_{\Gamma, {\mathbf g}}$ and $\mathscr{L}_{\Gamma, {\mathbf g}}$ proved in Thm.~(\ref{derived}) for $\mathbf{g}=0$. 

\begin{lem}
The Chekanov-Eliashberg DG-algebra $\mathscr{B}_{\Gamma, {\mathbf g}}$ of the Legendrian link $L_{\Gamma, {\mathbf g}}$ associated
    to a graph $\Gamma$ and an $s$-tuple $\mathbf{g}$ is quasi-isomorphic to the DG-algebra generated by 
\begin{itemize}
	\item generators $\tau_v$, $t_v^\pm$ for each vertex $v$,
	\item generators $\zeta_a$, $z_a^\pm$ for each arrow $a$ in $\Gamma \backslash T$, 
	\item generators $c_a$, $c_a^*$ for each arrow $a$ in $\Gamma$,
	\item generators $\alpha^\pm_{v,i}$,  $\beta^\pm_{v,i}$ for each vertex $v$, $i=1, \dots , g_v$
\end{itemize}
with $|\zeta_a|=|\tau_v|=-1$, $|t_v^\pm|=|z_a^\pm|=|c_a|=|c_a^*|=|\alpha^\pm_{v,i}|=|\beta^\pm_{v,i}|=0$ and nontrivial differentials
\begin{align*}
\d \zeta_a & = e_v+c_a^*c_a-z_a \\
 \d \tau_v&= \prod_{t(a)=v} (e_v+c_ac_a^*) -t_v \prod_{i=1}^{g_v} [\alpha_{v,i}, \beta_{v,i}]\prod_{s(a)=v, \ a\notin T} z_a \prod_{s(a)=v, \ a\in T}  (e_v + c_a^*c_a)
 \end{align*}
where $[x, y ]=xyx^{-1}y^{-1}$.
\end{lem}
\begin{proof}
We start with the presentation of $\mathscr{P}_{\Gamma, {\mathbf g}} \simeq \mathscr{B}_{\Gamma, {\mathbf g}}$ preceding the statement and (as in the proof of Lem.~(\ref{secondaries})) utilise an extended notion of stable tame isomorphism where all the elementary automorphisms are identity on the non-free generators $t_v^\pm, z_a^\pm, z^\pm_i$.
With the help of such automorphisms, we destabilise the pairs $(\xi_{v,j},x_{v,j})$ for all $j$. 
We do this in groups of four, starting with larger $j$ and moving our way down to $j=1$. 
Note that the treatment of the initial group is slightly different than the rest because of the unique feature of the lowest piece of Fig.~(\ref{thebeast}).
(In the rest of the proof, we fix $v$ and use the simplified notation $g$, $x_j, \xi_j$ and $z_j$ for $g_v$, $x_{v,j}, \xi_{v,j}$, and $z_{v,j}$, respectively.)

The first four automorphisms we apply are
$$x_{4g} \mapsto z_{2g-1}^{-1}(x_{4g}+e_v) , \ \ x_{4g-1} \mapsto x_{4g-1}+z_{2g}, \ \ x_{4g-2} \mapsto x_{4g-2}+z_{2g-1}$$
followed by
$$\xi_{4g-3} \mapsto \xi_{4g-3} + \xi_{4g-2} (x_{4g-1}+z_{2g}) + z_{2g-1} \xi_{4g-1} , \ \ \xi_{4g-4} \mapsto \xi_{4g-4} + x_{4g-3}z^{-1}_{2g-1}\xi_{4g}z_{2g-3}x_{4g-4}$$
then
$$x_{4g-3} \mapsto z^{-1}_{2g}(x_{4g-3}+z_{2g-1}z_{2g})$$
and finally
$$\xi_{4g-4}\mapsto \xi_{4g-4}+z^{-1}_{2g}\xi_{4g-3}z^{-1}_{2g-1}z_{2g-3}x_{4g-4}$$
After these automorphisms, the relevant differentials are changed as follows
\begin{align*}
\d \xi_{4g}  =&-x_{4g}, \ \ \d \xi_{4g-1}  =x_{4g-1}, \ \ \d \xi_{4g-2}  =x_{4g-2}, \ \ \d \xi_{4g-3} =-x_{4g-3} \\
\d \xi_{4g-4} &= e_v-z^{-1}_{2g}z_{2g-1}z_{2g}z^{-1}_{2g-1}z_{2g-3}x_{4g-4}=e_v-[z^{-1}_{2g}, z_{2g-1}]z_{2g-3}x_{4g-4}
\end{align*}
and hence we can destabilise the pairs $(\xi_i,x_i)$ for $i=4g-3, \dots , 4g$.

The next four automorphisms will help destabilise the next four pairs $(\xi_i,x_i)$ of generators and establish a pattern. 
These are
$$x_{4g-4} \mapsto z^{-1}_{2g-3} [ z_{2g-1},z^{-1}_{2g}](x_{4g-4}+e_v) , \ \ x_{4g-5} \mapsto x_{4g-5}+z_{2g-2}, \ \ x_{4g-6} \mapsto x_{4g-6}+z_{2g-3}$$
followed by
\begin{align*}
\xi_{4g-7} \mapsto  \xi_{4g-7} &+ \left(\xi_{4g-6} (x_{4g-5}+z_{2g-2})+ z_{2g-3} \xi_{4g-5}\right) z^{-1}_{2g-3} [ z_{2g-1},z^{-1}_{2g}](x_{4g-4}+e_v)  \\
& - z_{2g-3}z_{2g-2}z^{-1}_{2g-3}[ z_{2g-1},z^{-1}_{2g}]\xi_{4g-4}
\end{align*}
then
$$x_{4g-7} \mapsto z^{-1}_{2g-2}\left(x_{4g-7}+z_{2g-3}z_{2g-2}z^{-1}_{2g-3}[ z_{2g-1},z^{-1}_{2g}]\right)$$
and
$$\xi_{4g-8}\mapsto \xi_{4g-8}+z^{-1}_{2g-2}\xi_{4g-7}z_{2g-5}x_{4g-8}$$
These automorphisms result in the following differentials
\begin{align*}
\d \xi_{4g-4}  =&-x_{4g-4}, \ \ \d \xi_{4g-5}  =x_{4g-5}, \ \ \d \xi_{4g-6}  =x_{4g-6}, \ \ \d \xi_{4g-7} =-x_{4g-7} \\
\d \xi_{4g-8} &= e_v-[z^{-1}_{2g-2}, z_{2g-3}][ z_{2g-1},z^{-1}_{2g}]z_{2g-5}x_{4g-8}
\end{align*}
Applying $(g-2)$ additional sets of four automorphisms analogous to the last four above, we destabilise all the pairs $(\xi_i,x_i)$.
This turns the differential of $\tau_v$ into
$$\d \tau_v = \prod_{t(a)=v} (e_v+c_ac_a^*) -t_v P_C \prod_{s(a)=v, \ a\notin T} z_a \prod_{s(a)=v, \ a\in T}  (e_v + c_a^*c_a)$$
where $P_C$ is the product of commutators 
$$[z^{-1}_2, z_1] [z^{-1}_6, z_5] \cdots  [z^{-1}_{2g-2}, z_{2g-3}] [ z_{2g-1},z^{-1}_{2g}] \cdots [ z_{7},z^{-1}_{8}][ z_{3},z^{-1}_{4}] \text{\ if } \ g \ \text{is even, and}$$ 
$$[z^{-1}_2, z_1] [z^{-1}_6, z_5] \cdots  [z^{-1}_{2g}, z_{2g-1}] [ z_{2g-3},z^{-1}_{2g-2}] \cdots [ z_{7},z^{-1}_{8}][ z_{3},z^{-1}_{4}] \text{\ if } \ g \ \text{is odd}$$ 
In order to match this last differential with the expression in the statement (which makes it easier to see the quasi-isomorphism $\mathscr{B}_{\Gamma, {\mathbf g}} \simeq \mathscr{L}_{\Gamma, {\mathbf g}}$), we simply relabel the generators $z_i^\pm$ so that
$$\alpha_{v,i}=z_{4i-2}^{-1}, \ \ \alpha_{v,j}=z_{4g-4j+3}, \ \ \beta_{v,i}=z_{4i-3}, \ \ \beta_{v,j}=z^{-1}_{4g-4j+4}, 
\ \text{ for } \ 1\leq i\leq \frac{g+1}{2} <j \leq g$$
\end{proof}

\noindent {\it Proof of Thm.~(\ref{mainthm})}:  With the above lemma established, we proceed as in the proof of Thm.~(\ref{derived})
and consider the following isomorphism $\varphi: \mathscr{P}_{\Gamma, \mathbf{g}} \to
\mathscr{L}_{\Gamma,\mathbf{g}}$ of DG-algebras defined to be the identity on all primary generators except for 
$$\tau_v \mapsto \tau_v  + t_v  \prod_{i=1}^{g_v} [\alpha_{v,i}, \beta_{v,i}]\left( \sum_{i=1}^{\iota_v} z_{a_1}  \cdots z_{a_{i-1}}\zeta_{a_i} (e_v + c_{a_{i+1}}c^*_{a_{i+1}})\cdots (e_v + c_{a_{\iota_v}}c^*_{a_{\iota_v}})\hspace{-0.1cm}\right)\hspace{-0.2cm} \prod_{s(a)=v, a\in T} \hspace{-0.3cm} (e_v + c_a^*c_a) $$
where $\{ a_1, \ldots , a_{\iota_v}\} $ is the (ordered) set of arrows with $s(a_i)=v$ and $a_i\in \Gamma \setminus T$.
\qed

\begin{rmk}
Throughout the paper, we only considered plumbings with positive intersections. 
In the case of negative plumbings, the Chekanov-Eliashberg DG-algebra of the corresponding Legendrian link can be described by the same set of generators as above but the differential gets modified as follows
$$ \d \tau_v = \prod_{t(a)=v,  sgn(a)=+}(e_v + c_ac_a^*) - t_v \prod_{i=1}^{g_v} [\alpha_{v,i},\beta_{v,i}] \prod_{t(a)=v, sgn(a)=-}(e_v + c_ac_a^*) \prod_{s(a)=v} (e_v + c_a^*c_a), $$
where $sgn(a)$ denotes the sign of the plumbing associated to the arrow $a$. 
Importantly, the DG-algebra is no longer non-positively graded since $c_a$ and $c^*_a$ have gradings $\pm 1$ whenever $sgn(a)=-$.
\end{rmk}


\begin{thebibliography}{9999}

\bibitem{Ab} M. Abouzaid, 
A geometric criterion for generating the Fukaya category.
Publ. Math. Inst. Hautes \'Etudes Sci. 112 (2010), 191--240.

\bibitem{AS} M. Abouzaid, I. Smith,   Exact Lagrangians in plumbings. Geom. Funct. Anal. 22 (2012), no. 4, 785--831. 

		\bibitem{BK} R. Bezrukavnikov, M. Kapranov, Microlocal sheaves and quiver varieties. Ann. Fac. Sci. Toulouse Math. (6) 25 (2016), no. 2-3, 473--516.	

        \bibitem{boalch} P. Boalch, Global Weyl groups and a new theory of multiplicative quiver
            varieties. Geom. Topol. 19 (2015), no. 6, 3467--3536. 

			\bibitem{BEE} F. Bourgeois, T. Ekholm, Y. Eliashberg, Effect of Legendrian surgery. With an appendix by S. Ganatra and M. Maydanskiy, Geom. Topol. 16 (2012), no. 1, 301--389.
		
        \bibitem{CM} R. Casals, E. Murphy, Legendrian Fronts for Affine Varieties. Preprint,
            arXiv:1610.0697, 

		\bibitem{C} Y. Chekanov, Differential algebra of Legendrian links. Invent. Math. 150 (2002), no. 3, 441--483. 
		
		\bibitem{C-BS} W. Crawley-Boevey, P. Shaw, Multiplicative preprojective algebras, middle convolution and the Deligne-Simpson problem. Adv. Math. 201 (2006), no. 1, 180--208. 

             \bibitem{CDGG} B. Chantraine, G. Dimitroglou Rizell, P. Ghiggini, R. Golovko Geometric generation of the wrapped Fukaya category of Weinstein manifolds and sectors, preprint arXiv:1712.09126

		\bibitem{drinfeld} V. Drinfeld, DG quotients of DG categories. J. Algebra 272 (2004), no. 2, 643--691.
		
		  \bibitem{EkLe} T. Ekholm, Y. Lekili, Duality between Lagrangian and Legendrian invariants. Preprint, arXiv:1701.01284.
		  
		\bibitem{EkNg} T. Ekholm, L. Ng, Legendrian contact homology in the boundary of a subcritical Weinstein 4-manifold. J. Differential Geom. 101 (2015), no. 1, 67--157.	
		
		\bibitem{E} Y. Eliashberg, Invariants in contact topology. 
Proceedings of the International Congress of Mathematicians, Vol. II (Berlin, 1998). 
Doc. Math. 1998, Extra Vol. II, 327--338. 
			
		\bibitem{EL} T. Etg\"u, Y. Lekili, Koszul duality  patterns in Floer theory. Geom. Topol. 21 (2017), 3313--3389.
				
		\bibitem{G} R. Gompf, Handlebody construction of Stein surfaces. Ann. of Math. (2) 148 (1998), no. 2, 619--693.
		
		\bibitem{HKK}  F. Haiden, L. Katzarkov, M. Kontsevich, Flat surfaces and stability structures. Publ. Math. Inst. Hautes \'Etudes Sci. 126 (2017), 247--318.

        \bibitem{keating} A. Keating, Lagrangian tori in four-dimensional Milnor fibres. Geom. Funct. Anal. 25
            (2015), no. 6, 1822--1901.

\bibitem{miyachi} J.-I. Miyachi, Localization of Triangulated Categories and Derived Categories. Jour. of Algebra. 141 (1991), 463--483.

		\bibitem{NZ} D. Nadler, E. Zaslow , Constructible sheaves and the Fukaya category. J. Amer. Math. Soc. 22 (2009), no. 1, 233--286. 
		
		\bibitem{Ng} L. Ng, Computable Legendrian invariants. Topology 42 (2003), no. 1, 55--82.
		
              
        \bibitem{PTVV} T. Pantev, B. To\"en, M. Vaqui\'e, G. Vezzosi, Shifted symplectic structures. Publ. Math. Inst. Hautes \'Etudes Sci. 117 (2013), 271--328.     
   
     \bibitem{pascaleff} J. Pascaleff,  Floer cohomology in the mirror of the projective plane and a binodal cubic curve. Duke Math. J. 163 (2014), no. 13, 2427--2516. 
    
     \bibitem{Seidelgraded} P. Seidel, Graded Lagrangian submanifolds. Bull. Soc. Math. France 128
         (2000), no. 1, 103--149.

        \bibitem{Se} P. Seidel, Fukaya categories and Picard-Lefschetz theory. Zurich Lectures in Advanced Mathematics, European Mathematical Society (EMS), Z\"urich, 2008.

        \bibitem{ST} V. Shende, A. Takeda, Symplectic structures from topological Fukaya
            categories, Preprint, arXiv:1605.02721.

        \bibitem{yamakawa} D. Yamakawa, Geometry of multiplicative preprojective algebra. 
Int. Math. Res. Pap. IMRP 2008, Art. ID rpn008, 77pp.

\bibitem{yeung} W-K Yeung, Relative Calabi-Yau completions. Preprint, arXiv:1612.06352.
\end{thebibliography}
\end{document}